\documentclass[a4wide,11pt]{article}

\usepackage{amsmath,amssymb,amsfonts, amsthm}
\usepackage{verbatim}
\usepackage{mathrsfs}
\usepackage{url}
\usepackage{enumerate}
\usepackage{authblk}

\newtheorem{theorem}{Theorem}

\newtheorem{lemma}[theorem]{Lemma}

\newtheorem{remark}[theorem]{Remark}

\newcommand{\alp}{\mathrm{alph}}
\newcommand{\val}{\mathrm{val}}
\newcommand{\IRR}{\mathrm{IRR}}
\newcommand{\rank}{\mathrm{rank}}
\newcommand{\rhs}{\mathrm{rhs}}

\newcommand{\cA}{\mathcal{A}}
\newcommand{\cB}{\mathcal{B}}
\newcommand{\cG}{\mathcal{G}}
\newcommand{\cH}{\mathcal{H}}
\newcommand{\dG}{\mathbb{G}}
\newcommand{\dM}{\mathbb{M}}

\begin{document}

\title{Knapsack in graph groups, HNN-extensions and amalgamated products}

\author{Markus Lohrey}
\affil{Universit\"at Siegen, Germany\\
\texttt{lohrey@eti.uni-siegen.de}}

\author{Georg Zetzsche}
\affil{Technische Universit\"{a}t Kaiserslautern, Germany\\
\texttt{zetzsche@cs.uni-kl.de}}

\date{}

\maketitle

\begin{abstract}
It is shown that the knapsack problem, which was introduced by Myasnikov et al.
for arbitrary finitely generated groups, can be solved in {\sf NP} for graph
groups. This result even holds if the group elements are represented in a
compressed form by SLPs, which generalizes the classical {\sf NP}-completeness
result of the integer knapsack problem. We also prove general transfer results:
{\sf NP}-membership of the knapsack problem is passed on to finite
extensions, HNN-extensions over finite associated subgroups, and amalgamated
products with finite identified subgroups.
\end{abstract}

\section{Introduction}

In their paper \cite{MyNiUs14}, Myasnikov, Nikolaev, and Ushakov started the investigation of classical 
discrete optimization problems, which are classically formulated over the integers,
for arbitrary in general non-commutative groups.
Among other problems, they introduced for a finitely generated group $G$ 
the {\em knapsack problem} and the {\em subset sum problem}.
The input for the knapsack problem is a sequence of group elements $g_1, \ldots, g_k, g \in G$ (specified
by finite words over the generators of $G$) and it is asked whether there exists a solution
$(x_1, \ldots, x_k) \in \mathbb{N}^k$
of the equation $g_1^{x_1} \cdots g_k^{x_k} = g$. For the subset sum problem one restricts the solution
to $\{0,1\}^k$.
For the particular case $G = \mathbb{Z}$  (where the additive notation 
$x_1 \cdot g_1 + \cdots + x_k \cdot g_k = g$ is usually preferred)
these problems are {\sf NP}-complete if the numbers $g_1, \ldots, g_k,g$ are 
encoded in binary representation. For subset sum, this is a classical result from Karp's seminal paper \cite{Karp72} on 
{\sf NP}-completeness. Knapsack for integers is usually formulated in a more general form in the literature;
{\sf NP}-completeness of the above form (for binary encoded integers) was shown  in \cite{Haa11}, where 
the problem was called \textsc{multisubset sum}).\footnote{Note that if we ask for a solution $(x_1,\ldots, x_k)$
in $\mathbb{Z}^k$, then knapsack can be solved in polynomial time (even for binary encoded integers) by checking whether
$\mathrm{gcd}(g_1, \ldots, g_k)$ divides $g$.}
Interestingly, if we consider subset sum  for the group 
$G = \mathbb{Z}$, but encode the input numbers  $g_1, \ldots, g_k, g$ 
in unary notation, then the problem is in {\sf DLOGTIME}-uniform
$\mathsf{TC}^0$ (a small subclass of polynomial time and even of logarithmic space that captures the complexity
of multiplication of binary encoded numbers) \cite{ElberfeldJT11}, and 
the same holds for knapsack, since the instance $x_1 \cdot g_1 + \cdots + x_k \cdot g_k = g$ has a solution 
if and only if it has a solution with $x_i \leq  k \cdot (\max\{ g_1, \ldots, g_k,g \})^3$ \cite{Papadimitriou81}. This allows
to reduce unary knapsack to unary subset sum. See \cite{Jen95} for related results.

In \cite{MyNiUs14} the authors encode elements of the finitely generated group $G$ by words over the group generators
and their inverses. For $G = \mathbb{Z}$ this representation corresponds to the unary encoding of integers. 
Among others, the following results were shown in \cite{MyNiUs14}:
\begin{itemize}
\item Subset sum and knapsack can be solved in polynomial
time for every hyperbolic group. 
\item Subset sum for a virtually nilpotent group (a finite extension of a nilpotent group) can be solved in polynomial time. 
\item For the following groups, subset sum is {\sf NP}-complete (whereas the word problem can be solved in polynomial time):
free metabelian non-abelian groups of finite rank, the wreath product $\mathbb{Z} \wr \mathbb{Z}$, Thompson's group $F$,
and the Baumslag-Solitar group $\mathrm{BS}(1,2)$.
\end{itemize}
Further results on knapsack and subset sum have been recently obtained in \cite{KoeLoZe15}:
\begin{itemize}
\item For a virtually nilpotent group, subset sum belongs to {\sf NL} (nondeterministic logspace).
\item There is a nilpotent
group of class $2$ (in fact, a direct product of sufficiently many copies of the discrete Heisenberg group $H_3(\mathbb{Z})$),
for which knapsack is undecidable.
\item The knapsack problem for the discrete Heisenberg group $H_3(\mathbb{Z})$ is decidable. In particular, 
together with the previous point it follows that decidability
of knapsack is not preserved under direct products.
\item There is a polycyclic group with an {\sf NP}-complete subset sum problem.
\item The knapsack problem is decidable for all co-context-free groups.
\end{itemize}
The focus of this paper will be on the knapsack problem. We will prove that this problem can be solved in {\sf NP}
for every {\em graph group}. Graph groups are also known as right-angled Artin groups or free partially commutative 
groups. A graph group is specified by a finite simple graph. The vertices are the generators of the group, and two generators
$a$ and $b$ are allowed to commute if and only if $a$ and $b$ are adjacent. Graph groups somehow interpolate between 
free groups and free abelian groups and can be seen as a group counterpart of trace monoids (free partially commutative
monoids) that have been used for the specification of concurrent behavior. In combinatorial group theory, graph groups
are currently a hot topic, mainly because of their rich subgroup structure \cite{BesBr97,CrWi04,GhPe07}.    
To prove that knapsack belongs to {\sf NP} for a graph group, we proceed in two steps:
\begin{itemize}
\item We show that if an instance $g_1^{x_1} \cdots g_k^{x_k} = g$ has a solution in a graph group, then
it has a solution, where every $x_i$ is bounded exponentially in the input length (the total length of all words
representing the group elements $g_1, \ldots, g_k,g$). 
\item We then guess the binary encodings of numbers $n_1, \ldots, n_k$ that are bounded by the exponential bound
from the previous point and verify in polynomial time the identity $g_1^{n_1} \cdots g_k^{n_k} = g$. The latter
problem is an instance of the so called {\em compressed word problem} for a graph group. This is the classical word
problem, where the input group element is given succinctly by a so called {\em straight-line program} (SLP), which is a context-free
grammar that produces a single word (here, a word over the group generators and their inverses). An SLP 
with $n$ productions in Chomsky normal form can produce a string of length $2^n$. It has been shown in \cite{LoSchl07}
that the compressed word problem for a graph group can be solved in polynomial time, see also \cite{Loh14} for more details.
\end{itemize}
In fact, our proof yields a stronger result: First, it yields an {\sf NP} procedure for 
solving knapsack-like equations $h_0 g_1^{x_1} h_1 \cdots h_{k-1} g_k^{x_k} h_k = 1$, where some of the variables
$x_1, \ldots, x_k$ are allowed to be identical. We call such an equation an {\em exponent equation}.
Hence, we prove that solvability of exponent equations over a graph group belongs to {\sf NP}.

Second, we show that the latter result even holds, when the group elements $g_1, \ldots, g_k,h_0, \ldots, h_k$ 
are given succinctly by SLPs; we speak of {\em solvability of compressed exponent equations}. 
This is interesting, since the SLP-encoding of group elements corresponds in the case $G = \mathbb{Z}$
to the binary encoding of integers. Hence, membership in {\sf NP} for solvability of compressed exponent equations
over a graph group generalizes the classical {\sf NP} membership for knapsack (over $\mathbb{Z}$) to a much wider
class of groups.

Furthermore, we extend the class of groups for which solvability of knapsack
(resp. compressed exponent equations) can be checked in {\sf NP} by proving
general transfer results. Our first transfer result states that if $H$ is a
finite extension of $G$ and solvability of compressed exponent equations (or
knapsack) can be checked in {\sf NP} for $G$, then the same holds for $H$.
This provides such algorithms for the  large class of {\em virtually special
groups}. These are finite extensions of subgroups of graph groups.  Virtually
special groups recently played a major role in a spectacular breakthrough in
three-dimensional topology, namely the solution of the virtual Haken conjecture
\cite{Agol12}. In the course of this development it turned out that the class
of virtually special groups is extremely rich: It contains Coxeter groups
\cite{HagWi10}, one-relator groups with torsion \cite{Wis09}, fully residually
free groups \cite{Wis09}, and fundamental groups of  hyperbolic 3-manifolds
\cite{Agol12}.

We also prove transfer results for HNN-extensions and amalgamated products with
finite associated (resp. identified) subgroups in the case of the knapsack
problem. Such HNN-extensions and amalgamated products play a fundamental role
in combinatorial group theory~\cite{LySch77}.  For example, they appear in
Stallings' decomposition of groups with more than one end~\cite{Stal71} and in
the construction of virtually free groups~\cite{DiDu89}. Furthermore, they are known to
preserve a wide variety of structural and algorithmic properties (see Section~\ref{sec-transfer}).

 A side product of our proof is that the set of all solutions $(x_1, \ldots, x_k) \in \mathbb{N}^k$ of an exponent 
 equation $g_1^{x_1} \cdots g_k^{x_k} = g$ over a graph group is semilinear, and a semilinear representation 
 can be produced effectively. This seems to be true for many groups, e.g., for all co-context-free groups \cite{KoeLoZe15}.
 On the other hand, the discrete Heisenberg group $H_3(\mathbb{Z})$ is an example of a group for which 
 solvability of exponent equations is decidable but the set of all solutions of an exponent equation 
 is not semilinear; it is defined by a single quadratic Diophantine equation \cite{KoeLoZe15}.
 
 Finally, we complement our upper bounds by a new lower bound: Knapsack and subset sum are both {\sf NP}-complete
 for a direct product of two free groups of rank two ($F_2 \times F_2$). This group is the graph group corresponding to a cycle of length four.
 {\sf NP}-hardness already holds for the case that the input group elements are specified by words over the generators (for
 SLP-compressed words, {\sf NP}-hardness already holds for $\mathbb{Z}$) and the exponent variables
 are allowed to take values in $\mathbb{Z}$ (instead  $\mathbb{N}$). {\sf NP}-completeness of subset sum for 
 $F_2 \times F_2$  solves an open problem from \cite{FrenkelNU15}.
 
 \paragraph{\bf Related work.} The knapsack problem is a special case of the more general {\em rational subset membership problem}.
 A rational subset of a finitely generated monoid $M$ is the homomorphic image in $M$ of a regular language over the generators of $M$.
 In the rational subset membership problem for $M$ the input consists of a rational subset $L \subseteq M$ (specified by a finite automaton) and an element $m \in M$ and it is asked
 whether $m \in L$. It was shown in \cite{LohSte08} that the rational subset membership problem for a graph group $G$ is decidable if and only
 if the corresponding graph has (i) no induced cycle on four nodes (C4) and (ii) no induced path on four nodes (P4). 
 For the decidable cases, the precise complexity is open.

 Knapsack for $G$ can be also viewed as the question, whether a word equation $X_1 X_2 \cdots X_n =1$, where
 $X_1, \ldots, X_n$ are variables, together with constraints of the form $\{g^n \mid n \geq 0 \}$
 for the variables has a solution in $G$. Such a solution is a mapping $\varphi : \{X_1, \ldots, X_n\} \to G$ such that  $\varphi(X_1 X_2 \cdots X_n)$ evaluates to $1$ in $G$
 and all constraints are satisfied. For another class of constraints (so called normalized rational constraints, which do not cover constraints
 of the form $\{g^n \mid n \geq 0 \}$), solvability of general word equations 
 was shown to be decidable ($\mathsf{PSPACE}$-complete) for graph groups by Diekert and Muscholl \cite{DiMu06}. This result was extended in \cite{DiLo08IJAC} to a transfer theorem for 
 graph products. A graph product is specified by a finite simple graph, where every node is labelled with a group. The associated group is 
 obtained from the free product of all vertex groups by allowing elements from adjacent groups to commute. Note that decidability of knapsack
 is not preserved under graph products. It is even not preserved under direct products, see the above mentioned results from \cite{KoeLoZe15}.

\section{Words and straight-line programs}

For a word $w$ we denote with $\alp(w)$ the set of symbols occurring in $w$.
 The length of the word $w$ is $|w|$.
 
 A {\em straight-line program}, briefly {\em SLP}, is basically a context-free grammar that produces exactly one string.
To ensure this, the grammar has to be acyclic and deterministic (every variable has a unique production
where it occurs on the left-hand side). Formally, an SLP
is a tuple $\mathcal{G}  = (V,\Sigma,\rhs,S)$, where $V$ is a finite set of variables (or nonterminals),
$\Sigma$ is the terminal alphabet, 
$S \in V$ is the start variable, and $\rhs$ maps every variable to a right-hand side $\rhs(A) \in (V \cup \Sigma)^*$.
We require that there is a linear order $<$ on $V$ such that $B < A$, whenever $B \in N \cap \alp(\rhs(A))$.
Every variable $A \in V$ derives to a unique string $\val_{\cG}(A)$ 
by iteratively replacing variables by the corresponding right-hand sides, starting with $A$.
Finally, the string derived by $\cG$ is $\val(\cG) = \val_{\cG}(S)$.

Let  $\cG  = (V,\Sigma,\rhs,S)$ be an SLP.
The {\em size} of $\cG$ is $|\cG| = \sum_{A \in V} |\rhs(A)|$, i.e.,  the total length of all right-hand sides.
A simple induction shows that for every SLP $\cG$ of size $m$ one has
$|\val(\cG)| \leq \mathcal{O}(3^{m/3}) \subseteq 2^{O(n)}$ \cite[proof of Lemma~1]{CLLLPPSS05}. On the other hand, it is straightforward to define an SLP
$\cH$ of size $2n$ such that $|\val(\cH)| \geq 2^n$. 
This justifies to see an SLP $\cG$ as a compressed representation of the string $\val(\cG)$, 
and exponential compression rates can be achieved in this way.
More details on SLPs can be found in the survey \cite{Loh12survey}.
 
 \section{Knapsack and exponent equations}
 
 We assume that the reader has some basic knowledge concerning (finitely generated) groups, see e.g. \cite{LySch77}
 for further details. Let $G$ be a finitely generated group, and let $A$ be a finite generating set for $G$. Then,
 elements of $G$ can be represented by finite words over the alphabet $A^{\pm 1} = A \cup A^{-1}$.

An {\em exponent equation} over $G$ is an equation of the
form $$v_0 u_1^{x_1} v_1 u_2^{x_2}  v_2 \cdots u_n^{x_n} v_n = 1,$$ 
where
$u_1, u_2, \ldots, u_n, v_0, v_1, \ldots, v_n \in G$ are group elements that
are given by finite words over the alphabet $A^{\pm 1}$
and $x_1, x_2, \ldots, x_n$
are not necessarily distinct variables. Such an exponent equation is solvable
if there exists a mapping $\sigma : \{x_1, \ldots,x_n\} \to \mathbb{N}$ such that
$v_0 u_1^{\sigma(x_1)} v_1 u_1^{\sigma(x_2)}  v_2 \cdots u_n^{\sigma(x_n)} v_n = 1$
in the group $G$. 
{\em Solvability of  exponent equations over $G$} is the following computational problem:

\smallskip
\noindent
{\bf Input:} An exponent equation $E$ over $G$ (where elements of $G$ are specified by words over the 
group generators and their inverses).

\smallskip
\noindent
{\bf Question:} Is $E$ solvable?

\smallskip
\noindent
The {\em knapsack problem} for the group $G$ is the restriction of solvability of  exponent equations over $G$
to exponent equations of the form $u_1^{x_1} u_2^{x_2} \cdots u_n^{x_n} u^{-1} = 1$, or, equivalently 
$u_1^{x_1} u_2^{x_2} \cdots u_n^{x_n} = u$, where the exponent variables $x_1, \ldots, x_n$ have to be
pairwise different.

We will also study a compressed version of exponent equations over $G$, where elements of $G$ are given by SLPs
over $A^{\pm 1}$. A {\em compressed exponent equation} is an exponent equation
$v_0 u_1^{x_1} v_1 u_2^{x_2}  v_2 \cdots u_n^{x_n} v_n = 1$, where the group elements 
$u_1, u_2, \ldots, u_n, v_0, v_1, \ldots, v_n \in G$ are given by SLPs over the terminal alphabet
$A^{\pm 1}$.  The sum of the sizes of these SLPs is the size of the 
compressed exponent equation.

Let us define {\em solvability of compressed  exponent equations over $G$} as the following computational problem:

\smallskip
\noindent
{\bf Input:} A compressed exponent equation $E$ over $G$.

\smallskip
\noindent
{\bf Question:} Is $E$ solvable?

\smallskip
\noindent
The {\em compressed knapsack problem} for $G$ is defined analogously. Note that with this terminology, the classical knapsack problem for binary encoded integers
is the compressed knapsack problem for the group $\mathbb{Z}$. The binary encoding of an integer can be easily transformed into an SLP over the alphabet
$\{a, a^{-1}\}$ (where $a$ is a generator of $\mathbb{Z}$) and vice versa. Thereby the number of bits in the binary encoding and the size of the SLP are linearly
related.

It is a simple observation  that the decidability and complexity of solvability of  (compressed) exponent equations over $G$ as well as  
the (compressed) knapsack problem for $G$ does not depend on the chosen finite generating set for the group $G$.  
Therefore, we do not have to mention the generating set explicitly in these problems.

\begin{remark} \label{remark}
Since we are dealing with a group, one might also allow solution mappings $\sigma :  \{x_1, \ldots,x_n\} \to \mathbb{Z}$
to the integers. But this variant of solvability of  (compressed) exponent equations (knapsack, respectively) can be reduced to the above
version, where $\sigma$ maps to $\mathbb{N}$, by simply replacing a power $u_i^{x_i}$ by $u_i^{x_i} (u^{-1}_i)^{y_i}$, where
$y_i$ is a fresh variable.
\end{remark}
The goal of this paper is to prove the decidability of solvability of exponent equations for so called graph groups.
We actually prove that solvability of compressed exponent equations for a graph group belongs to {\sf NP}.
Graph groups will be introduced in the next section.

\section{Traces and graph groups}

Let $(A,I)$ be a finite simple graph. In other words, the edge relation $I \subseteq A \times A$ is 
irreflexive and symmetric. It is also called the {\em independence relation}, and $(A,I)$ is called
an {\em independence alphabet}.
We consider the monoid $\dM(A, I) = A^*/\!\!\equiv_I$,
where $\equiv_I$ is the smallest congruence relation on the free monoid $A^*$ that contains all pairs
$(ab, ba)$ with $a,b \in A$ and $(a,b) \in I$. This monoid is called a {\em trace monoid} or {\em partially
commutative free monoid}. Elements of $\dM(A, I)$ are called 
{\em Mazurkiewicz traces} or simply {\em traces}. The trace represented by the word $u$ is denoted
by $[u]_I$, or simply $u$ if no confusion can arise. For a language $L \subseteq A^*$
we denote with $[L]_I = \{ u \in A^* \mid \exists v \in L : u \equiv_I v \}$ its {\em partially commutative
closure}.  The length of the trace $[u]_I$ is
$|[u]_I| = |u|$ and its alphabet is $\alp([u]_I) = \alp(u)$. It is easy to see that these
definition do not depend on the concrete word that represents the trace $[u]_I$.
For subsets $B,C \subseteq A$ we write $B I C$ for $B \times C \subseteq I$. If $B = \{a\}$
we simply write $a I C$.
For traces $s,t$ we write $s I t$ for $\alp(s) I \alp(t)$. 
The empty trace $[\varepsilon]_I$ is the identity element of the monoid $\dM(A,I)$ 
and is denoted by $1$.
A trace $t$ is connected if we cannot factorize $t$ as $t = uv$ with $u \neq 1 \neq v$
and $u I v$.

A trace $t \in \dM(A,I)$ can be visualized by its \emph{dependence 
graph} $D_t$. 
To define $D_t$, choose an arbitrary word $w = a_1 a_2 \cdots a_n$,
$a_i \in A$, with $t = [w]_I$ and define 
$D_t = ( \{1,\ldots,n\}, E, \lambda)$, where
$E =  \{ (i,j) \mid i<j, (a_i,a_j) \in D\}$ and $\lambda(i) = a_i$.
If we identify isomorphic dependence graphs, then this 
definition is independent of the chosen word representing $t$.
Moreover, the mapping $t \mapsto D_t$ is injective.
As a consequence of the representation of traces
by dependence graphs, 
one obtains Levi's lemma for traces, see e.g. 
\cite[p. 74]{DieRoz95}, which is one of the fundamental
facts in trace theory. The formal statement is as follows.

\begin{lemma} \label{lemma-levi}
Let $u_1, \ldots, u_m, v_1, \ldots, v_n \in \dM(A,I)$. Then 
\[ u_1u_2 \cdots u_m = v_1 v_2 \cdots  v_n\] if and only if
there exist $w_{i,j} \in \dM(A,I)$ $(1 \leq i \leq m$, $1 \leq j \leq n)$ such that
\begin{itemize}
\item $u_i = w_{i,1}w_{i,2}\cdots w_{i,n}$ for every $1 \leq i \leq m$,
\item $v_j = w_{1,j}w_{2,j}\cdots w_{m,j}$ for every $1 \leq j \leq n$, and
\item $(w_{i,j}, w_{k,\ell})\in I$ if $1 \leq i < k \leq m$ and $n \geq j > \ell \geq 1$.
\end{itemize}
\end{lemma}
\noindent
The situation in the lemma will be visualized by a 
diagram of the following kind. The $i$--th column
corresponds to $u_i$, the $j$--th row
corresponds to $v_j$, and the intersection of the $i$--th
column and the $j$--th row represents $w_{i,j}$.
Furthermore $w_{i,j}$ and $w_{k,\ell}$ are independent
if one of them is left-above the other one. 
\begin{center}
  \begin{tabular}{c||c|c|c|c|c|}\hline
  $v_n$  & $w_{1,n}$ & $w_{2,n}$ & $w_{3,n}$ & \dots  & $w_{m,n}$ \\ \hline
  \vdots & \vdots    & \vdots    & \vdots    & \vdots & \vdots    \\ \hline
  $v_3$  & $w_{1,3}$ & $w_{2,3}$ & $w_{3,3}$ & \dots  & $w_{m,3}$ \\ \hline
  $v_2$  & $w_{1,2}$ & $w_{2,2}$ & $w_{3,2}$ & \dots  & $w_{m,2}$ \\ \hline
  $v_1$  & $w_{1,1}$ & $w_{2,1}$ & $w_{3,1}$ & \dots  & $w_{m,1}$ \\ \hline\hline
         & $u_1$     & $u_2$     & $u_3$     & \dots  & $u_m$
  \end{tabular}
  \end{center}
A consequence of Levi's Lemma is that
trace monoids are cancellative, i.e., $usv=utv$ implies $s=t$ for all
traces $s,t,u,v\in\dM(A,I)$. 

For a trace $u \in \dM(A,I)$ let $\rho(u)$ be the number of prefixes of $u$.
We will use the following statement from \cite{BeMaSa89}.

\begin{lemma} \label{lemma-prefixes}
Let $u \in \dM(A,I)$ be a trace of length $n$. Then $\rho(u) \in O(n^\alpha)$, where $\alpha$
is the size of a largest clique of the complementary graph $(A,I)^c = (A, (A \times A) \setminus I)$.
\end{lemma}
\noindent
With an independence alphabet $(A,I)$ we associate the group 
$$
\dG(A,I) = \langle A \mid ab=ba \ ( (a,b) \in I) \rangle.
$$
Such a group is called a {\em graph group}, or
{\em right-angled Artin group}\footnote{This term comes from the fact that right-angled Artin groups
are exactly  the Artin groups corresponding to right-angled Coxeter groups.}, or {\em free partially commutative group}.
Here, we use the term graph group.
Graph groups received a lot of attention in group theory during the last few years, mainly
due to their rich subgroup structure \cite{BesBr97,CrWi04,GhPe07}, and their relationship to low dimensional
topology (via so called virtually special groups) \cite{Agol12,HagWi10,Wis09}.
We represent elements of $\dG(A,I)$ by traces over an extended independence alphabet.
For this, let $A^{-1} = \{ a^{-1} \mid a \in A\}$ be a disjoint copy of the alphabet $A$, and
let $A^{\pm 1} = A \cup A^{-1}$. We define $(a^{-1})^{-1} ¥= a$ and for a word
$w = a_1 a_2 \cdots a_n$ with $a_i \in A^{\pm 1}$ we define $w^{-1} = a^{-1}_n \cdots a^{-1}_2 a^{-1}_1$.
This defines an involution (without fixed points) on $(A^{\pm 1})^*$.
We extend the independence relation $I$ to $A^{\pm 1}$ by
$(a^x, b^y) \in I$ for all $(a,b) \in I$ and $x,y \in \{-1,1\}$. Then, there is a canonical surjective
morphism $h : \dM(A^{\pm 1},I) \to \dG(A,I)$ that maps every symbol $a \in A^{\pm 1}$ to the 
corresponding group element. Of course, $h$ is not injective, but we can easily define a subset
$\IRR(A^{\pm 1},I) \subseteq \dM(A^{\pm 1}, I)$ of {\em irreducible traces} such that $h$ restricted
to $\IRR(A^{\pm 1},I)$ is bijective. The set $\IRR(A¥^{\pm 1},I)$ consists of all traces $t \in \dM(A^{\pm 1},I)$ 
such that $t$ does not contain a factor $[aa^{-1}]_I$ with $a \in A^{\pm 1}$, i.e., there do not exist 
$u, v \in  \dM(A^{\pm 1}, I)$ and $a \in A^{\pm 1}$ such that in $\dM(A^{\pm 1}, I)$ we have a factorization $t = u [aa^{-1}]_I v$.
For every trace $t$ there exists a corresponding {\em irreducible normal form} that is obtained by
removing from $t$ factors $[aa^{-1}]_I$ with $a \in A^{\pm 1}$ as long as possible. It can be shown 
that this reduction process is terminating (which is trivial since it reduces the length) and confluent (in \cite{KuLo05ijac}
a more general confluence lemma for graph products of monoids is shown).
Hence, the  irreducible normal form of $t$ does not depend on the concrete order of reduction steps.
For a group element $g \in \dG(A,I)$ we denote with $|g|$ the length of the unique trace $t \in \IRR(A^{\pm 1},I)$
such that $h(t)=g$.

For a trace $t = [u]_I$ ($u \in (A^{\pm 1})^*$) we can define
$t^{-1} = [u^{-1}]_I$. This is well-defined, since $u \equiv_I v$ implies $u^{-1} \equiv_I v^{-1}$.
The following lemma will be important, see \cite[Lemma 23]{DiMu06}:

\begin{lemma}\label{lemma-product-IRR}
Let $s,t \in \IRR(A^{\pm 1},I)$. Then there exist unique factorizations $s = u p$, $t = p^{-1} v$ such that
$uv \in \IRR(A^{\pm 1},I)$. Hence, $uv$ is the  irreducible normal form of $st$.
\end{lemma}

\section{Factorizations of powers}

Based on Levi's lemma we prove in this section a factorization result for powers of a connected
trace. We start with the case that we factorize such a power into two factors.

\begin{lemma}  \label{lemma-u^n=xy}
Let $u \in \dM(A,I) \setminus \{1\}$ be a connected trace. Then,  for 
all $x \in \mathbb{N}$ and all traces $y_1,y_2$ the following two statements are equivalent:
\begin{enumerate}[(i)]
\item $u^x = y_1 y_2$
\item There exist $l,k,c \in \mathbb{N}$ and traces $s,p$ such that:
$y_1 = u^{l} s$, $y_2 = p u^k$, $s p = u^c$, $l+k+c = x$, and $c \leq |A|$.
\end{enumerate}
\end{lemma}  

\begin{proof}
That (ii) implies (i) is clear. It remains to prove that (i) implies (ii).
Assume that  $u^x = y_1 y_2$ holds.
The case that $x \leq |A|$ is trivial. Hence, assume that 
$x \geq |A|+1$. We apply Levi's lemma (Lemma~\ref{lemma-levi}) to the identity $u^x = y_1 y_2$:
\begin{center}
  \begin{tabular}{c||c|c|c|c|c|c|c|}\hline
  $y_2$  & $u_{1,2}$ & $u_{2,2}$ & $u_{3,2}$  & $u_{4,2}$ & $\cdots$  &  $u_{x-1,2}$  & $u_{x,2}$ \\ \hline
  $y_1$  & $u_{1,1}$ & $u_{2,1}$ & $u_{3,1}$  & $u_{4,1}$ & $\cdots$  &  $u_{x-1,1}$  & $u_{x,1}$ \\ \hline\hline
          & $u$          & $u$          & $u$            & $u$          & $\cdots$  & $u$ & $u$
  \end{tabular}
  \end{center}
  Let $A_i = \alp(u_{1,2} \cdots u_{i,2})$. Then $A_i \subseteq A_{i+1}$.
  If $A_1 = \emptyset$ then $u_{1,2} = 1$ and we can go to Case 2 below.
  Otherwise, assume that $A_1 \neq \emptyset$.
  In that case there must exist $1 \leq i \leq |A|$ such that $A_i = A_{i+1}$, which implies
  $\alp(u_{i+1,2}) \subseteq A_i$.
  Since  $u_{i+1,1} I (u_{1,2} \cdots u_{i,2})$ we also have $u_{i+1,1} I u_{i+1,2}$.
  Since $u$ is connected, we have 
  $u_{i+1,1} = 1$ or $u_{i+	1,2} = 1$. 
  We can therefore distinguish the following two cases:
  
\medskip
\noindent
{\em Case 1.} There exists $1 \leq i \leq |A|+1$ such that 
$u_{i,1}=1$. Then $u_{i,2}=u$, which implies $u_{j,1} = 1$ for all $j > i$ (since $u_{i,2} I u_{j,1}$):
\begin{center}
  \begin{tabular}{c||c|c|c|c|c|c|c|c|c|}\hline
  $y_2$  & $u_{1,2}$ & $u_{2,2}$ & $\cdots$ &    $u_{i-1,2}$  & $u$  &  $u$ & $\cdots$  &  $u$  & $u$ \\ \hline
  $y_1$  & $u_{1,1}$ & $u_{2,1}$ & $\cdots$ &    $u_{i-1,1}$  & $1$  &  $1$ & $\cdots$  &  $1$  & $1$ \\ \hline\hline
         & $u$          & $u$           & $\cdots$ &   $u$              &  $u$ &   $u$ & $\cdots$  & $u$ & $u$
  \end{tabular}
  \end{center}
  Let $s = u_{1,1} u_{2,1} \cdots u_{i-1,1}$
  and $p = u_{1,2} u_{2,2} \cdots u_{i-1,2}$.
 Thus, $y_1 = u^0s$, $y_2 = p u^{x-i+1}$ and $s p = u^{i-1}$  with $i-1 \leq |A|$,
 and the conclusion of the lemma holds.

\medskip
\noindent
{\em Case 2.} There exists $1 \leq i \leq |A|+1$ such that 
$u_{i,2}=1$. Then, $u_{j,2} = 1$ for all $j < i$ (since $u_{i,1} = u$ and $u_{j,2} I u_{i,1}$):
\begin{center}
  \begin{tabular}{c||c|c|c|c|c|c|c|c|c|}\hline
  $y_2$  & $1$ & $1$ & $\cdots$ &    $1$  & $1$  &  $u_{i+1,2}$ & $\cdots$  &  $u_{x-1,1}$  & $u_{x,2}$ \\ \hline
  $y_1$  & $u$ & $u$ & $\cdots$ &    $u$  & $u$  &  $u_{i+1,1}$ & $\cdots$  &  $u_{x-1,1}$  & $u_{x,1}$ \\ \hline\hline
          & $u$          & $u$           & $\cdots$ &   $u$              &  $u$ &   $u$          & $\cdots$  & $u$ & $u$
  \end{tabular}
  \end{center}
 Let $y'_1 = u_{i+1,1} \cdots u_{x,1}$.
 Hence, $u^{x-i} = y'_1 y_2$.
 We can use induction to get factorizations
$y'_1 = u^{l} s$,
$y_2 = p u^k$,
and $s p = u^c$ with $c \leq |A|$ and $k+l+c = x-i$.
 Finally, we have $y_1 = u^i y'_1 = u^{i+l} s$, 
 which shows the conclusion of the lemma.
\end{proof}
Now we lift Lemma~\ref{lemma-u^n=xy} to an arbitrary number of factors.

\begin{lemma} \label{lemma-simplify-factorization-power}
Let $u  \in \dM(A,I) \setminus \{1\}$ be a connected trace and $m \in \mathbb{N}$, $m \geq 2$. 
Then, for all $x \in \mathbb{N}$ and traces $y_1, \ldots, y_m$ the following two statements are equivalent:
\begin{enumerate}[(i)]
\item $u^x = y_1 y_2 \cdots y_m$. 
\item  There exist traces $p_{i,j}$ $(1 \leq j < i \leq m)$, $s_i$ $(1 \leq i \leq m)$  and
numbers $x_i, c_j \in \mathbb{N}$ $(1 \leq i \leq m$, $1 \leq j \leq m-1)$
such that:
\begin{itemize}
\item $y_i = (\prod_{j=1}^{i-1} p_{i,j})  u^{x_i} s_i$ for all $1 \leq i \leq m$,
\item $p_{i,j} I p_{k,l}$ if $j < l < k < i$ and  $p_{i,j} I (u^{x_k} s_k)$ if $j <  k < i$
\item $s_m = 1$ and for all $1 \leq j < m$, $s_j \prod_{i=j+1}^m p_{i,j} = u^{c_j}$
\item $c_j \leq |A|$ for all $1 \leq j \leq m-1$,
\item $x = \sum_{i=1}^m x_i + \sum_{i=1}^{m-1} c_i$.
\end{itemize}
\end{enumerate}
\end{lemma}

\begin{proof}
Let us first show that (ii) implies (i). Assume that (ii) holds.
Then we get
$$
y_1 y_2 \cdots y_m =  \prod_{i=1}^m \bigg( (\prod_{j=1}^{i-1} p_{i,j})  u^{x_i} s_i \bigg).
$$
The independencies  $p_{i,j} I p_{k,l}$ for $j < l < k < i$ and  $p_{i,j} I (u^{x_k} s_k)$ for $j <  k < i$
yield
\begin{eqnarray*}
&& \prod_{i=1}^m \bigg( \big(\prod_{j=1}^{i-1} p_{i,j}\big)  u^{x_i} s_i \bigg) \\ 
&=& u^{x_1} s_1 p_{2,1} \cdots p_{m,1} 
u^{x_2} s_2 p_{3,2} \cdots p_{m,2} u^{x_3} s_3 \cdots  u^{x_{m-1}} s_{m-1} p_{m,m-1} u^{x_m} s_m  \\
&=& u^{x_1} u^{c_1} u^{x_2} u^{c_2} u^{x_3} \cdots u^{c_{m-1}} u^{x_m} = u^x .
\end{eqnarray*}
We now prove that (i) implies (ii) by induction on $m$. 
So, assume that $u^x = y_1 y_2 \cdots y_m$.
The case $m=2$ follows directly from 
Lemma~\ref{lemma-u^n=xy}. Now assume that $m \geq 3$.
By Lemma~\ref{lemma-u^n=xy} there exist factorizations $y_1 = u^{x_1} s_1$, $y_2 \cdots y_m = p_1 u^{x'}$,
and $s_1p_1 = u^{c_1}$ with $c_1 \leq |A|$ and $x_1 + x'  + c_1 = x$.
Levi's lemma applied to $y_2 \cdots y_m = p_1 u^{x'}$ gives the following diagram:
\begin{center}
  \begin{tabular}{c||c|c|c|c|c|c|c|}\hline
  $y_m$  & $p_{m,1}$ &  \multicolumn{6}{c|}{$y'_m$}  \\ \hline
   $\vdots$ &  $\vdots$ &  \multicolumn{6}{c|}{$\vdots$}  \\ \hline
  $y_3$  & $p_{3,1}$ &  \multicolumn{6}{c|}{$y'_3$} \\ \hline
  $y_2$  & $p_{2,1}$ &   \multicolumn{6}{c|}{$y'_2$}  \\ \hline\hline  
          & $p_1$          & $u$          & $u$            & $u$    & \dots  & $u$ & $u$
  \end{tabular}
  \end{center}
  There exist $y'_i$ with $y_i = p_{i,1} y'_i$ ($2 \leq i \leq m$),
  $y'_2 \cdots y'_m = u^{x'}$,
   and $y'_j I p_{i,1}$ for $j < i$.
  By induction on $m$ we get factorizations 
  $$y'_i = \prod_{j=2}^{i-1} p_{i,j} u^{x_i} s_i$$ for $2 \leq i \leq m$ such that
  for all $2 \leq j < i \leq m$:
\begin{itemize}
\item $p_{i,j} I p_{k,l}$ if $j < l < k < i$ and  $p_{i,j} I (u^{x_k} s_k)$ if $j <  k < i$,
\item $s_m = 1$ and for all $2 \leq j < m$,  $s_j \prod_{i=j+1}^m p_{i,j} = u^{c_j}$ for some $c_j \leq |A|$,
\item $x' = \sum_{i=2}^m x_i + \sum_{i=2}^{m-1} c_i$.
\end{itemize}
Since $y'_j I p_{i,1}$ for $j < i$ we get
$p_{i,1} I p_{j,k}$ for $1 < k < j <i$ and  $p_{i,1} I u^{x_j} s_j$ for $1 < j < i$.
Finally, we have
$$
s_1 \prod_{i=2}^m p_{i,1}  =  s_1 p_1 = u^{c_1} \\
$$
and
$$x = x_1 + c_1 +  x'  = x_1  + c_1 + \sum_{i=2}^m x_i + \sum_{i=2}^{m-1} c_i = \sum_{i=1}^m x_i + \sum_{i=1}^{m-1} c_i .
$$
This proves the lemma.
\end{proof}

\begin{remark} \label{remark-simplify-factorization-power}
In Section~\ref{sec-main}
we will apply Lemma~\ref{lemma-simplify-factorization-power} in order to replace an equation $u^x = y_1 y_2 \cdots y_m$,
(where $x,y_1, \ldots, y_m$ are variables and $u$ is a concrete connected trace) by an equivalent disjunction. Note that the length of all factors $p_{i,j}$ and $s_i$
above is bounded by $|A| \cdot |u|$. Hence, one can guess these traces as well 
as the numbers $c_j \leq |A|$  (the guess results in a big disjunction). We can also 
guess which of the numbers $x_i$ are zero and which are greater than zero. After these guesses we can verify the independences
$p_{i,j} I p_{k,l}$ ($j < l < k < i$) and  $p_{i,j} I (u^{x_k} s_k)$ ($j <  k < i$), and the identities
$s_m = 1$, $s_j \prod_{i=j+1}^m p_{i,j} = u^{c_j}$ ($1 \leq j < m$). If one of them does not hold, the specific guess does
not contribute to the disjunction. In this way,
we can replace the equation $u^x = y_1 y_2 \cdots y_m$ by a big disjunction of formulas of the form
$$
\exists x_i > 0 \; (i \in K):
x = \sum_{i\in K}^m x_i + c \wedge \bigwedge_{i \in K} y_i = p_i  u^{x_i} s_i  \wedge \bigwedge_{i \in [1,m] \setminus K}  y_i = p_i s_i ,
$$
where $K \subseteq [1,m]$,
$c \leq |A| \cdot (m-1)$ and the $p_i, s_i$ are concrete traces of length 
at most  $|A| \cdot (m-1) \cdot |u|$. The number of disjuncts in the disjunction will not be important for our purpose.
\end{remark}

\section{Automata for partially commutative closures}

In this section, we present several automata constructions that are well-known from the theory
of recognizable trace languages \cite[Chapter~2]{Die90lncs}. For our purpose we need upper bounds on the size (the size
of an automaton is its number of states) of 
the constructed automata. In our specific situation we can obtain better bounds than those obtained
from the known constructions. Therefore, we present the constructions in detail.

Let us fix an independence alphabet $(A,I)$ and let 
$\mathcal{A} = (Q, A, \Delta, q_{0}, F)$ be a nondeterministic finite automaton (NFA) over the alphabet $A$, where $\Delta \subseteq Q \times A \times Q$
is the transition relation, $q_0 \in Q$ is the initial state and $F \subseteq Q$ is the set of final states.
Then, $\mathcal{A}$ is 
an $I$-diamond NFA if 
for all $(a,b) \in I$ and all transitions $(p,a,q), (q,b,r) \in \Delta$ there exists
a state $q'$ such that $(p,b,q'), (q',a,r) \in \Delta$.
For an $I$-diamond automaton we have $L(\mathcal{A}) = [L(\mathcal{A})]_I$.
The NFA $\mathcal{A}$ is {\em memorizing} if
(i) every state is accessible from the initial state $q_0$ and 
(ii)  there is a mapping $\alpha : Q \to 2^A$ such that
for every word $w \in A^*$, if $q_0 \stackrel{w}{\longrightarrow}_{\mathcal{A}} q$, then 
$\alpha(q) = \alp(w)$.

\begin{lemma} \label{lemma-product-I-diamond}
Let $\mathcal{A}_1$ and $\mathcal{A}_2$ be $I$-diamond NFA and let $n_i$ be the 
number of states of $\mathcal{A}_i$. Assume that $\mathcal{A}_2$ is memorizing.
Then there exists an $I$-dia-mond NFA for  $[L(\mathcal{A}_1) L(\mathcal{A}_2)]_I$ with 
$n_1 \cdot n_2$ many states.
\end{lemma}

\begin{proof}
Let $\mathcal{A}_i = (Q_i, A, \Delta_i, q_{0,i}, F_i)$ for $i \in \{1,2\}$. 
Let $\alpha_2 : Q_2 \to 2^A$ be the map witnessing the fact that $\mathcal{A}_2$ is memorizing.
Then,  let 
$$
\mathcal{A} = (Q_1 \times Q_2, A, \Delta, \langle q_{0,1}, q_{0,2}\rangle, F_1 \times F_2),
$$
where 
\begin{eqnarray*}
\Delta & = & \{ ( \langle p_1, p_2\rangle, a, \langle q_1, p_2\rangle) \mid (p_1, a, q_1) \in \Delta_1, a I \alpha_2(p_2) \} \cup \\
           &   & \{ ( \langle p_1, p_2\rangle, a, \langle p_1, q_2\rangle) \mid (p_2, a, q_2) \in \Delta_2\} .
\end{eqnarray*}
This indeed defines an $I$-diamond NFA. 

We show that the following two
statements are equivalent for all $w \in A^*$, $p_1 \in Q_1$, and $p_2 \in Q_2$:
\begin{enumerate}[(i)]
\item  $\langle q_{0,1}, q_{0,2}\rangle \stackrel{w}{\longrightarrow}_{\mathcal{A}} \langle p_1, p_2\rangle$
\item There are $w_1, w_2 \in A^*$ such that $w \equiv_I w_1 w_2$, $q_{0,1} \stackrel{w_1}{\longrightarrow}_{\mathcal{A}_1} p_1$, and
$q_{0,2} \stackrel{w_2}{\longrightarrow}_{\mathcal{A}_2} p_2$.
\end{enumerate}
This clearly implies that $L(\mathcal{A}) = [L(\mathcal{A}_1) L(\mathcal{A}_2)]_I$.

Let us first prove that (i) implies (ii). The case $w = \varepsilon$ is clear. Hence, let $w = w'a$.
Then there exist $p'_1 \in Q_1$, $p'_2 \in Q_2$ such that
$$
\langle q_{0,1}, q_{0,2}\rangle \stackrel{w'}{\longrightarrow}_{\mathcal{A}} \langle p'_1, p'_2\rangle
 \stackrel{a}{\longrightarrow}_{\mathcal{A}} \langle p_1, p_2\rangle .
$$
By induction, there exists a factorization $w' \equiv_I w'_1 w'_2$ such that
$q_{0,1} \stackrel{w'_1}{\longrightarrow}_{\mathcal{A}_1} p'_1$ and
$q_{0,2} \stackrel{w'_2}{\longrightarrow}_{\mathcal{A}_2} p'_2$.
Note that $\alp(w'_2) = \alpha_2(p'_2)$. There are two cases:

\medskip
\noindent
{\em Case 1.} $p'_1 \stackrel{a}{\longrightarrow}_{\mathcal{A}_1} p_1$,  $p_2 = p'_2$, and $a I \alpha_2(p'_2)$.
Thus, $a I w'_2$. We get $w = w'a \equiv_I w'_1 w'_2 a \equiv_I (w'_1 a) w'_2$. 
Let $w_1 = w'_1 a$ and $w_2 = w'_2$. We get $q_{0,1} \stackrel{w_1}{\longrightarrow}_{\mathcal{A}_1} p_1$
and $q_{0,2} \stackrel{w_2}{\longrightarrow}_{\mathcal{A}_2} p_2$.

\medskip
\noindent
{\em Case 2.} $p'_2 \stackrel{a}{\longrightarrow}_{\mathcal{A}_2} p_2$  and  $p_1 = p'_1$.
Let $w_1 = w'_1$ and $w_2 = w'_2 a$. Thus, $w = w'a \equiv_I w'_1 w'_2 a = w_1 w_2$.
Moreover, we have
$q_{0,1} \stackrel{w_1}{\longrightarrow}_{\mathcal{A}_1} p_1$ and 
$q_{0,2} \stackrel{w_2}{\longrightarrow}_{\mathcal{A}_2} p_2$.

\medskip
\noindent
Let us now prove that (ii) implies (i).  Assume that $w \equiv_I w_1 w_2$, $q_{0,1} \stackrel{w_1}{\longrightarrow}_{\mathcal{A}_1} p_1$, and
$q_{0,2} \stackrel{w_2}{\longrightarrow}_{\mathcal{A}_2} p_2$. We have to show that 
$\langle q_{0,1}, q_{0,2}\rangle \stackrel{w}{\longrightarrow}_{\mathcal{A}} \langle p_1, p_2\rangle$.
But since $\mathcal{A}$ is an $I$-diamond NFA, it suffices to show that
$\langle q_{0,1}, q_{0,2}\rangle \stackrel{w_1 w_2}{\longrightarrow}_{\mathcal{A}} \langle p_1, p_2\rangle$,
which follows directly from the assumption and the definition of $\mathcal{A}$ (note that 
$\alpha_2(q_{0,2}) = \emptyset$).
This concludes the proof of the lemma.
\end{proof}
\noindent
In general, for a regular language $L \subseteq A^*$, the partially commutative closure 
$[L]_I$ is not regular. For instance, if $A = \{a,b\}$ and $a I b$, then $[(ab)^*]_I$ consists
of all words with the same number of $a$'s as $b$'s. On the other hand, it is well known
that if $u$ is a connected trace, then $[u^*]_I$ is regular (a more general result, known
as Ochmanski's theorem holds in fact, see e.g. \cite[Section~2.3]{Die90lncs}). For our purpose we need an upper
on the size of an $I$-diamond NFA for $[u^*]_I$ (with $u$ connected). Recall that $\rho(u)$ is the number of 
different prefixes of the trace $u$.

\begin{lemma} \label{lemma-anca}
Let $u \in \dM(A,I)  \setminus \{1\}$ be connected. There is a memorizing $I$-diamond NFA for $[u^*]_I$ of size 
$2 \cdot \rho(u)^{|A|}$.
\end{lemma}

\begin{proof}
The following construction can be found in  \cite[Proposition~5]{MuPe99} for the more general
case of the partially commutative closure of a so called loop-connected automaton. We present
the construction in our simplified situation, since the NFA gets slightly smaller. 

We first define a non-memorizing $I$-diamond NFA $\mathcal{A}$ for $[u^*]_I$ of size 
$\rho(u)^{|A|}$. Then, we show that by adding an additional bit to all states, we can get
a memorizing $I$-diamond NFA $\mathcal{A}$ for $[u^*]_I$ of size 
$2 \cdot \rho(u)^{|A|}$. The idea for the construction of $\mathcal{A}$ is implicitly contained in the proof
of Lemma~\ref{lemma-u^n=xy}: Assume that the automaton wants to read a word of the form $u^x$ and 
a prefix $y_1$ is already read. Then $y_1$ must be of the form $u^k s$, where $s$ is a prefix of $u^c$ for some
$c \leq |A|$. The prefix $s$ must be of the form $u_1 u_2 \cdots u_c$ such that if $u = u_i v_i$, then
$v_i I u_j$ if $i<j$.  The state of the NFA stores the tuple $(u_1, u_2, \ldots, u_c)$.

Define $\mathcal{A} = (Q, A, \Delta, q_0, F)$, where 
$Q$ is the set of all tuples $(u_1, u_2, \ldots, u_{c})$ of traces such that there exist $v_1, \ldots, v_c \in \dM(A,I)$ 
with $u = u_i v_i$ (since $\dM(A,I)$ is cancellative, the $v_i$ are uniquely determined by the $u_i$), $u_i \neq 1$, $v_i \neq 1$, and $v_i I u_j$ if $i<j$. 
Note that we must have $c \leq |A|$: If $c > |A|$, then there exist $i  \leq |A|$ such that
$\alp(v_1 \cdots v_i) = \alp(v_1 \cdots v_{i+1})$. Hence, $\alp(v_{i+1}) \subseteq \alp(v_1 \cdots v_i)$.
Since $(v_1 \cdots v_i) I u_{i+1}$ we get $u_{i+1} I v_{i+1}$ which contradicts $u_{i+1} \neq 1 \neq v_{i+1}$ and the fact
that $u$ is connected.

Since $u_i \neq 1$ for all $i$, 
we can encode a state $(u_1, u_2, \ldots, u_{c}) \in Q$ by the tuple
$(u_1, u_2, \ldots, u_{c}, 1, \ldots, 1)$ of length $|A|$. This implies that the number of states of $\mathcal{A}$ is
bounded by $\rho(u)^{|A|}$.  Note that if $|u|=1$, then the empty tuple $()$ is the only state.

The transitions of $\mathcal{A}$ are defined as follows, where $(u_1, u_2, \ldots, u_{c}) \in Q$:
\begin{enumerate}[(a)]
\item $() \stackrel{a}{\longrightarrow}_{\mathcal{A}} ()$ if $u=a \in A$,
\item $(u_1, u_2, \ldots, u_{c}) \stackrel{a}{\longrightarrow}_{\mathcal{A}}  (u_2, \ldots, u_{c})$  if $c > 0$, $u_1 a = u$ and $a I (u_2 \cdots u_{c})$,
\item $(u_1, u_2, \ldots, u_{c}) \stackrel{a}{\longrightarrow}_{\mathcal{A}}  (u_1, \ldots, u_{i}, a, u_{i+1},\ldots, u_c)$ if  
$a I (u_{i+1} \cdots u_{c})$ and $(u_1, \ldots, u_{i}, a, u_{i+1},\ldots, u_c) \in Q$,
\item $(u_1, u_2, \ldots, u_{c}) \stackrel{a}{\longrightarrow}_{\mathcal{A}}  (u_1, \ldots, u_{i-1}, u_i a, u_{i+1},\ldots, u_c)$ if 
$a I (u_{i+1} \cdots u_{c})$ and $(u_1, \ldots, u_{i-1}, u_i a, u_{i+1},\ldots, u_c) \in Q$.
\end{enumerate}
The initial state as well as the final state is the empty tuple $()$. It is easy to check that this is indeed an $I$-diamond NFA.

We claim that for every state $(u_1, \ldots, u_c) \in Q$ and every $w \in A^*$ the following two statements are equivalent
(which shows that $L(\mathcal{A}) = [u^*]_I$):
\begin{enumerate}[(i)]
\item $() \stackrel{w}{\longrightarrow}_{\mathcal{A}} (u_1, \ldots, u_c)$
\item $w \equiv_I u^k u_1 \cdots u_c$ for some $k \geq 0$
\end{enumerate}
Let us first show by induction on $|w|$ that (i) implies (ii). The case $w=\varepsilon$ is clear. So, assume that $w = w'a$.
There must exist a state $(u'_1, \ldots, u'_d) \in Q$ such that
$$
() \stackrel{w'}{\longrightarrow}_{\mathcal{A}} (u'_1, \ldots, u'_d) \stackrel{a}{\longrightarrow}_{\mathcal{A}}  (u_1, \ldots, u_c).
$$
By induction, we get $w' \equiv_I u^{\ell} u'_1\cdots u'_d$ for some $\ell \geq 0$.
The definition of the transitions of $\mathcal{A}$ implies that
$w = w'a \equiv_I u^{\ell} u'_1\cdots u'_d a \equiv_I u^k u_1 \cdots u_c$, where $k \in \{\ell, \ell+1\}$.

For the direction from (ii) to (i) assume that $w \equiv_I u^k u_1 \cdots u_c$ for some $k \geq 0$.
We have to show that $() \stackrel{w}{\longrightarrow}_{\mathcal{A}} (u_1, \ldots, u_c)$. Since
$\mathcal{A}$ is an $I$-diamond NFA, it suffices to show that 
$() \xrightarrow{u^k u_1 \cdots u_c}_{\mathcal{A}} (u_1, \ldots, u_c)$. But this follows directly
from the definition of $\mathcal{A}$.

To make $\mathcal{A}$ memorizing, we first keep only those states that are accessible from the initial state $()$.
Then, we add an extra bit to every state that indicates whether we have already seen a completed occurrence of $u$.
Thus, the new set of states is $Q \times \{0,1\}$, the initial state is the pair $((), 0)$, and the final states are $((),0)$ and $((),1)$. 
The transitions operate on the $Q$-component as  for $\mathcal{A}$. The $\{0,1\}$-component is copied except for a transition 
$(q,a,p) \in \Delta$ of type
($b$). This transition gives us the transitions $((q,0),a,(p,1))$ and $((q,1),a,(p,1))$.
Then, we can define the $\alpha$-mapping by
$$\alpha((u_1, \ldots, u_c),i) = \bigcup_{j=1}^c \alp(u_j) \cup \alp(u^i).$$
The resulting NFA is still an $I$-diamond NFA.
\end{proof}
\noindent
A direct consequence of Lemma~\ref{lemma-product-I-diamond} and \ref{lemma-anca} is:

\begin{lemma}
Let $p, u, s \in \dM(A,I)$ with $u \neq 1$ connected. There is an NFA for $[p u^* s]_I$ of size $2 \cdot \rho(p) \cdot \rho(s) \cdot \rho(u)^{|A|}$.
\end{lemma}

\begin{proof}
We first construct an $I$-diamond NFA for $p$ (which is identified here with the set of words $\{ u \in A^* \mid p = [u]_I \}$) 
with $\rho(p)$ many states by taking the set of 
all prefixes of $p$ as states. Then, we construct a memorizing $I$-diamond NFA for $[u^*]_I$ with $2 \cdot \rho(u)^{|A|}$ states
using Lemma~\ref{lemma-anca}.
By Lemma~\ref{lemma-product-I-diamond} we get an $I$-diamond automaton for $[p u^*]_I$  with $2 \cdot \rho(p) \cdot \rho(u)^{|A|}$
many states. Finally, we construct an $I$-diamond NFA for $s$ with $\rho(s)$ many states by taking the set of 
all prefixes of $s$ as states. This NFA is also memorizing. Hence, we can apply Lemma~\ref{lemma-product-I-diamond}  to get
an NFA for $[p u^*s]_I$  with $2 \cdot \rho(p) \cdot \rho(s) \cdot \rho(u)^{|A|}$
many states. 
\end{proof}
\noindent
The main lemma from this section that will be needed later is:

\begin{lemma} \label{lemma-connected-star}
Let $p,q,u,v,s,t \in \dM(A,I)$ with $u\neq 1$ and $v \neq 1$ connected.  Let $m = \max\{ \rho(p), \rho(q), \rho(s), \rho(t) \}$
and $n = \max \{ \rho(u), \rho(v) \}$.
Then the set 
$$L(p,u,s,q,v,t) := \{ (x,y) \in \mathbb{N} \times \mathbb{N} \mid p u^x s = q v^y t \}$$
is semilinear and is a union of 
$O(m^8 \cdot n^{4 |A|})$ many
linear sets of the form $\{ (a+bz, c+dz) \mid z \in \mathbb{N} \}$ with
$a,b,c,d \in O(m^8 \cdot n^{4 |A|})$. 
\end{lemma}

\begin{proof}
By Lemma~\ref{lemma-anca} there exists an NFA for $[p u^* s]_I$ of size
$$k = 2 \cdot \rho(p) \cdot \rho(s)  \cdot \rho(u)^{|A|} \leq
2 \cdot m^2 \cdot n^{|A|}   
$$
and an NFA for $[q v^* t]_I$ of size
$$
\ell = 2 \cdot \rho(q) \cdot \rho(t) \cdot \rho(v)^{|A|} \leq
2 \cdot m^2 \cdot n^{|A|}.
$$
Then, we obtain an NFA $\mathcal{A}$ for $L = [p u^* s]_I \cap [q v^* t]_I$ 
with  $k \cdot \ell$ states. 
We are only interested in the length of words from $L$. Hence, we replace 
in $\mathcal{A}$ every transition label by the symbol $a$. The resulting NFA
$\mathcal{B}$ is defined over a unary alphabet. Let $P = \{ n \mid a^n \in L(\mathcal{B}) \}$.
By \cite[Theorem~1]{To09ipl}, the set $P$ can be written as a union 
$$
P = \bigcup_{i=1}^r \{  b_i + c_i \cdot z \mid z \in \mathbb{N} \}
$$
with $r \in O(k^2 \ell^2) \subseteq O(m^8 \cdot n^{4 |A|})$ and $b_i, c_i \in O(k^2 \ell^2)\subseteq O(m^8 \cdot n^{4 |A|})$. 
For every $1 \leq i \leq r$ and $z \in \mathbb{N}$ there must exist a pair $(x,y) \in \mathbb{N} \times \mathbb{N}$
such that 
$$
b_i + c_i \cdot z = |ps| + |u| \cdot x = |qt| + |v| \cdot y.
$$
In particular, $b_i \geq |ps|$, $b_i \geq |qt|$, $|u|$ divides $b_i-|ps|$ and $c_i$,
and $|v|$ divides $b_i-|qt|$ and $c_i$.
We get:
\begin{multline*}
L(p,u,s,q,v,t) =  \bigcup_{i=1}^r \bigg\{ \bigg( \frac{b_i-|ps|}{|u|} + \frac{c_i}{|u|} \cdot z, \frac{b_i-|qt|}{|v|} + \frac{c_i}{|v|} \cdot z \bigg) \; \bigg| \bigg. \; z \in \mathbb{N}\bigg\}
\end{multline*}
This shows the lemma.
\end{proof}

\section{Linear Diophantine equations}

We will also need a bound on the norm of a smallest vector in a certain kind of 
semilinear sets. We will easily obtain this bound from a result by Zur Gathen and Sieveking \cite{ZGSi78}.

\begin{lemma} \label{lemma-linear-diophantine}
Let  $A \in \mathbb{Z}^{n \times m}$, $\overline{a} \in \mathbb{Z}^n$, 
$C \in \mathbb{N}^{k \times m}$, $\overline{c} \in \mathbb{N}^k$.
Let $\beta$ be an upper bound for the absolute value of all entries
in $A$, $\overline{a}$, $C$, $\overline{c}$.
The set
\begin{equation} \label{eq-final-equation2}
L = \{ C \overline{z} + \overline{c} \mid   \overline{z} \in \mathbb{N}^m, A \overline{z} = \overline{a} \} \subseteq \mathbb{N}^k
\end{equation}
is semilinear. Moreover, if $L \neq \emptyset$ then $L$ contains a vector with all entries bounded by 
$\beta + n! \cdot m \cdot (m+1) \cdot \beta^{n+1}$.
\end{lemma}

\begin{proof}
Semilinearity of $L$ is clear since the set is Presburger-definable.
For the size bound,
we use a result by Zur Gathen and Sieveking \cite{ZGSi78} to bound the size of a smallest positive solution of the system
$A \overline{z} = \overline{a}$. 
Let $A \in \mathbb{Z}^{n \times m}$, $B \in \mathbb{Z}^{p \times m}$, 
$\overline{a} \in \mathbb{Z}^{n \times 1}$, $\overline{b} \in \mathbb{Z}^{p \times 1}$.
Let $r = \rank(A)$, and
$\displaystyle s = \rank\begin{pmatrix} A \\ B \end{pmatrix}$.
Let $M$ be an upper bound on the absolute values of all $(s-1)\times (s-1)$-
or $(s \times s)$-subdeterminants of the $(n+p)\times (m+1)$-matrix
$\displaystyle \begin{pmatrix} A & \overline{a} \\ B & \overline{b} \end{pmatrix}$, which are formed
with at least $r$ rows from the matrix $(A \;\, \overline{a})$. Then by the main result of \cite{ZGSi78}, the system
$A \overline{z} = \overline{a}$, $B \overline{z} \geq \overline{b}$ has an integer solution if and only if it has
an integer solution $\overline{z}$ such that the absolute value of
every entry of $\overline{z}$ is bounded by $(m+1) M$.

In our situation, we set $p = m$, $B$ is the $m$-dimensional identity matrix,
and $\overline{b}$ is the vector with all entries equal to zero (then 
$B \overline{z} \geq \overline{b}$ expresses that all entries of $z$ are positive).
Since $\begin{pmatrix} A \\ B \end{pmatrix}$  is an $(n+m) \times m$-matrix
we get $\displaystyle s = \rank\begin{pmatrix} A \\ B \end{pmatrix} \leq m$.
We claim that
the absolute values of all $(s \times s)$-subdeterminants (and also all $(s-1)\times (s-1)$-subdeterminants)
of the matrix $\displaystyle \begin{pmatrix} A & \overline{a} \\ B & \overline{b} \end{pmatrix}$ are bounded
by $n! \cdot \beta^n$. To see this, select $s$ rows and $s$ columns from $\displaystyle \begin{pmatrix} A & \overline{a} \\ B & \overline{b} \end{pmatrix}$
and consider the resulting submatrix $D$. Recall Leibniz' formula for the determinant (where $S_s$ is the set of all permutations of $\{1,\ldots,s\}$):
$$
\mathrm{det}(D) = \sum_{\sigma \in S_s} \mathrm{sgn}(\sigma) \prod_{i=1}^s D[i,\sigma(i)].
$$
Assume that the rows $1, \ldots, s_1$ ($s_1 \leq n$) of $D$ are from the $n \times (m+1)$-submatrix $(A, \overline{a})$. The remaining 
($s_2 := s-s_1$ many) rows $s_1+1, \ldots, s$ 
of $D$ are  from $(B, \overline{b})$. If one of the rows $s_1+1, \ldots, s$ of $D$ only contains zeros, then
$\mathrm{det}(D) = 0$. Otherwise, since $B$ is the identity matrix and $\overline{b}$ is the zero vector,
each of the rows $s_1+1, \ldots, s$ contains a unique $1$;
all other entries are zero. That means that every permutation $\sigma \in S_s$ that gives a non-zero 
contribution to $\mathrm{det}(D)$ must take fixed values on $s_1+1, \ldots, s$. For the values of $\sigma$ on 
the rows $1, \ldots, s_1$, only $s_1 \leq n$  many values remain. Hence, at most 
$s_1! \leq n!$ many permutations contribute a non-zero value to $\mathrm{det}(D)$.  Moreover, every such contribution is bounded by $\beta^{s_1} \leq \beta^n$,
which gives the bound $n! \cdot \beta^n$ on $\mathrm{det}(D)$.
It follows that if $A \overline{z} = \overline{a}$ has a positive solution, then it has a positive
solution where every entry is bounded by $(m+1) \cdot n! \cdot \beta^n$.

By substituting every entry of $\overline{z}$ by $(m+1) n! \cdot \beta^n$ in 
$C \overline{z} + \overline{c}$, it follows that if the set $L$ in \eqref{eq-final-equation2} is non-empty, 
then it contains a vector with all entries bounded by 
$\beta + n! \cdot m \cdot (m+1) \cdot \beta^{n+1}$.
\end{proof}

\section{Exponent equations in graph groups} \label{sec-main}

The aim of this section is to prove the following two statements, where $G$ is a fixed graph group:
\begin{itemize}
\item The set of solutions of an exponent equation over $G$ is (effectively) semilinear.
\item  Solvability of  compressed exponent equations over $G$ belongs to {\sf NP}.
\end{itemize}
We start with some definitions. As usual, we fix an independence alphabet $(A,I)$. In the following
we will consider reduction rules on sequences of traces. For better readability we separate the consecutive 
traces in such a sequence by commas.
Let $u_1, u_2, \ldots, u_n \in \IRR(A^{\pm 1},I)$ be irreducible traces.
The sequence  $u_1, u_2, \ldots, u_n$ is {\em $I$-freely reducible}
if the sequence $u_1, u_2, \ldots, u_n$ can be reduced to the empty sequence $\varepsilon$
by the following rules:
\begin{itemize}
\item $u_i, u_j \to u_j, u_i$  if $u_i I u_j$ 
\item $u_i, u_j \to \varepsilon$  if $u_i = u_j^{-1}$ in $\dG(A,I)$  
\item $u_i \to \varepsilon$ if $u_i = \varepsilon$.
\end{itemize}
A concrete sequence of these rewrite steps leading to the empty sequence
is a {\em reduction} of  the sequence $u_1, u_2, \ldots, u_n$. Such a reduction can be seen as
a witness for the fact that $u_1 u_2 \cdots u_n = 1$ in
$\dG(A,I)$. On the other hand, $u_1 u_2 \cdots u_n = 1$ does not necessarily imply
that $u_1, u_2, \ldots, u_n$ has a reduction. For instance, the sequence $a^{-1}, ab, b^{-1}$
has no reduction. But we can show that every sequence which multiplies to $1$ in $G$ can be refined
(by factorizing the elements of the sequence) such that the resulting refined sequence has a reduction.
For getting an {\sf NP}-algorithm, it is important to bound the length of the refined sequence exponentially
in the length of the initial sequence.

\begin{lemma} \label{lemma-reduction}
Let $n \geq 2$ and $u_1, u_2, \ldots, u_n \in \IRR(A^{\pm 1},I)$. If $u_1 u_2 \cdots u_n = 1$ in
$\dG(A,I)$, then there exist factorizations
$u_i = u_{i,1}  \cdots u_{i,k_i}$ such that the sequence
$$
u_{1,1}, \ldots, u_{1,k_1}, \; u_{2,1}, \ldots, u_{2,k_2}, \; \ldots, u_{n,1}, \ldots, u_{n,k_n}
$$
is $I$-freely reducible. Moreover, $\sum_{i=1}^n k_i \leq 2^{n}  - 2$.
\end{lemma}  

\begin{proof}
We prove the lemma by induction on $n$. 
The case $n=2$ is trivial (we must have $u_2 = u_1^{-1}$). If $n \geq 3$ then by Lemma~\ref{lemma-product-IRR} we can factorize $u_1$ and $u_2$ as 
$u_1 = p s$ and $u_2 = s^{-1} t$ such that $v := pt$ is irreducible.
Hence, $v u_3 \cdots u_n = 1$ in $\dG(A,I)$. By induction, we obtain 
factorizations $pt = v = v_{1}  \cdots v_{k}$ and 
$u_i = v_{i,1} \cdots v_{i,k_i}$ ($3 \leq i \leq n$) such that the sequence
\begin{equation} \label{sequence-IH}
v_{1}, \ldots, v_{k}, \; v_{3,1}, \ldots, v_{3,k_3}, \ldots, v_{n,1}, \ldots, v_{n,k_n}
\end{equation}
is $I$-freely reducible.  Moreover, 
$$
k + \sum_{i=3}^n k_i \leq 2^{n-1} - 2 .
$$
By applying Levi's lemma  to the identity $pt = v_{1} v_{2} \cdots v_{k}$,
we obtain factorizations $v_{i} = u_{i,1} u_{i,2}$ such that
$p = u_{1,1} \cdots u_{k,1}$, $t = u_{1,2} \cdots u_{k,2}$, and $u_{i,2} I u_{j,1}$ for $1 \leq i < j \leq k$.

Fix a concrete reduction of the sequence \eqref{sequence-IH}.
We now consider the following sequence
\begin{equation} \label{sequence-main}
u_{1,1}, \ldots, u_{k,1}, s, \; s^{-1}, u_{1,2}, \ldots, u_{k,2}, \; \tilde{v}_{3,1}, \ldots, \tilde{v}_{3,k_3}, \ldots, \tilde{v}_{n,1}, \ldots, \tilde{v}_{n,k_n},
\end{equation}
where the subsequence $\tilde{v}_{i,j}$ is $u_{l,2}^{-1}, u_{l,1}^{-1}$ if $v_{i,j}$ cancels against $v_l$ in our fixed reduction 
of \eqref{sequence-IH} (which, in particular implies that $v_{i,j} = v_l^{-1} = u_{l,2}^{-1} u_{l,1}^{-1}$).
Otherwise (i.e., if $v_{i,j}$ does not cancel against any $v_l$ in our fixed reduction),
we set $\tilde{v}_{i,j} = v_{i,j}$.

Note that $u_{1,1} \cdots u_{k,1} s = ps = u_1$, $s^{-1} u_{1,2} \cdots u_{k,2} = s^{-1} t = u_2$ and the concatenation
of all traces in $\tilde{v}_{i,1}, \ldots, \tilde{v}_{i,k_i}$ is $u_i$ for $3 \leq i \leq n$. Hence, it remains to 
show that the sequence \eqref{sequence-main} is $I$-freely reducible.
First of all,  $u_{1,1}, \ldots, u_{k,1}, s, s^{-1}, u_{1,2}, \ldots, u_{k,2}$ reduces to
$u_{1,1}, \ldots, u_{k,1}, u_{1,2}, \ldots, u_{k,2}$, which can be rearranged to 
$u_{1,1}, u_{1,2}, u_{2,1}, u_{2,2}, \ldots, u_{k,1}, u_{k,2}$ using the fact that 
$u_{i,2} I u_{j,1}$ for $1 \leq i < j \leq k$. Finally, the sequence 
$$
u_{1,1} u_{1,2}, u_{2,1} u_{2,2}, \ldots, u_{k,1} u_{k,2},  \tilde{v}_{3,1}, \ldots, \tilde{v}_{3,k_3}, \ldots, \tilde{v}_{n,1}, \ldots, \tilde{v}_{n,k_n}
$$
is $I$-freely reducible. The definition of $\tilde{v}_{i,j}$ allows to basically apply the fixed reduction of \eqref{sequence-IH} to this sequence.

The number of traces in the sequence \eqref{sequence-main} can be estimated as 
$$
2 k + 2 + 2 \cdot  \sum_{i=3}^n k_i \leq  2 \cdot (2^{n-1} - 2) + 2 = 2^n-2.
$$
This concludes the proof of the lemma.
\end{proof}
\noindent
We now come to the main technical result of this paper.
Let $\alpha \leq |A|$ be the size 
of a largest clique of the complementary graph $(A,I)^c = (A, (A \times A) \setminus I)$.

\begin{theorem} \label{thm-main-technical}
Let $u_1, u_2, \ldots, u_n \in  \dG(A,I) \setminus \{1\}$, 
$v_0, v_1, \ldots, v_{n} \in \dG(A,I)$ and let $x_1, \ldots, x_n$ be variables (we may have $x_i = x_j$ for $i \neq j$) ranging
over $\mathbb{N}$. Then, the set of solutions of the exponent equation 
$$v_0 u_1^{x_1} v_1 u_2^{x_2} v_2 \cdots u_n^{x_n} v_{n} = 1$$ is semilinear. Moreover, 
if there is a solution, then there is a solution with $x_i \in 
O((\alpha n)! \cdot  2^{2\alpha^2 n(n+3)} \cdot \mu^{8\alpha(n+1)} \cdot  \nu^{8 \alpha |A| (n+1)})$,
where
\begin{itemize}
\item $\mu \in  O(|A|^\alpha \cdot 2^{2 \alpha^2 n} \cdot \lambda^\alpha)$, 
\item $\nu  \in O(\lambda^\alpha)$, and
\item $\lambda = \max\{ |u_1|, |u_2|, \ldots, |u_n|, |v_0|, |v_1|, \ldots, |v_{n}|\}$.
\end{itemize}
\end{theorem}

\begin{proof}
Let us choose irreducible traces for $u_1, u_2, \ldots, u_n, v_0, v_1, \ldots, v_n$; we denote these traces with the same letters as the group elements.
A trace $u$ is called cyclically reduced if there do not exist $a \in A^{\pm 1}$ and $v$ such that $u = a v a^{-1}$. 
For every trace there exist unique traces $p, w$ such that $u = p w p^{-1}$  and $w$ is cyclically reduced (since the reduction relation 
$a^{-1} x a \to x$ is terminating and confluent).
These traces $p$ and $w$ can be computed in polynomial time. Note that for a cyclically reduced irreducible trace $w$, every power
$w^n$ is irreducible.  By replacing every $u_i^{x_i}$ by $p_i w_i^{x_i} p_i^{-1}$ with $u_i = p_i w_i p_i^{-1}$ and $w_i$ 
cyclically reduced, we can assume that all $u_i$ are cyclically reduced and irreducible.
In case one of the traces $u_i$ is not connected, we can write $u_i$ as $u_i = u_{i,1} u_{i,2}$ with $u_{i,1} I u_{i,2}$ and $u_{i,1} \neq 1 \neq u_{i,2}$. Thus, we can 
replace the power $u_i^{x_i}$ by $u_{i,1}^{x_i} u_{i,2}^{x_i}$. Note that $u_{i,1}$ and $u_{i,2}$ are still irreducible and cyclically reduced.
By doing this, the number $n$ from the theorem
multiplies by at most $\alpha$ (which is the maximal number of pairwise independent letters).
In order to keep the notation simple we still use the letter $n$ for the number of $u_i$, but at the end
of the proof we have to multiply $n$ by $\alpha$ in the derived bound.
Hence, for the further proof we can assume that all $u_i$ are connected, irreducible and cyclically reduced. 
Let $\lambda$ be the maximal length of one of the traces $u_1, u_2, \ldots, u_n, v_0, v_1, \ldots, v_n$, which does not 
increase by the above preprocessing.

We now apply Lemma~\ref{lemma-reduction} to the equation
\begin{equation} \label{initial-equation}
v_0 u_1^{x_1} v_1 u_2^{x_2} v_2 \cdots u_n^{x_n} v_n = 1,
\end{equation}
where every $u_i^{x_i}$ is viewed as a single factor.
Note that by our preprocessing, all factors $u_1^{x_1}, u_2^{x_2}, \ldots, u_n^{x_n}, v_0, \ldots, v_n$ are irreducible
(for all choices of the $x_i$). By taking a big disjunction over (i) all possible factorizations 
of the $2n+1$ factors $u_1^{x_1}, u_2^{x_2}, \ldots, u_n^{x_n}, v_0, \ldots, v_n$ into totally at most 
$2^{2n+1}  - 2$ factors and (ii) all possible reductions of the resulting refined factorization of 
$v_0 u_1^{x_1} v_1 u_2^{x_2} v_2 \cdots u_n^{x_n} v_n$, it follows that \eqref{initial-equation}
is equivalent to a disjunction of statements of the following form: There exist
traces $y_{i,1}, \ldots, y_{i,k_i}$ ($1 \leq i \leq n$) and $z_{i,1}, \ldots, z_{i,l_i}$ ($0 \leq i \leq n$) such that
\begin{enumerate}[(a)]
\item $u_i^{x_i} = y_{i,1} \cdots y_{i,k_i}$  ($1 \leq i \leq n$)
\item $v_i = z_{i,1} \cdots z_{i,l_i}$ ($0 \leq i \leq n$)
\item $y_{i,j} I y_{k,l}$ for all $(i,j,k,l) \in J_1$
\item $y_{i,j} I z_{k,l}$ for all $(i,j,k,l) \in J_2$
\item $z_{i,j} I z_{k,l}$ for all $(i,j,k,l) \in J_3$
\item $y_{i,j} = y_{k,l}^{-1}$ for all $(i,j,k,l) \in M_1$
\item $y_{i,j} = z_{k,l}^{-1}$ for all $(i,j,k,l) \in M_2$
\item $z_{i,j} = z_{k,l}^{-1}$ for all $(i,j,k,l) \in M_3$
\end{enumerate}
Here, the numbers $k_i$ and $l_i$ sum up to at most $2^{2n+1}-2$ (hence, some $k_i$ 
can be exponentially large, whereas $l_i$ can be bound by the length of $v_i$, which is at most $\lambda$). The tuple 
sets $J_1, J_2, J_3$ collect all  independences between the factors $y_{i,j}$, $z_{k,l}$
that are necessary to carry out the chosen reduction of the refined left-hand side in \eqref{initial-equation}.
Similarly, the tuple sets $M_1, M_2, M_3$ tell us which of the factors  $y_{i,j}$, $z_{k,l}$ cancels
against which of the factors $y_{i,j}$, $z_{k,l}$ in our chosen reduction of the refined left-hand side in \eqref{initial-equation}.
Note that every factor $y_{i,j}$ (resp., $z_{k,l}$) appears in exactly one of the identities (f), (g), (h) (since
in the reduction every factor cancels against another unique factor). 

Next, we simplify our statements. Since the $v_i$ are concrete traces (of length at most $\lambda$), we can take a disjunction over
all possible factorizations $v_i = v_{i,1} \cdots v_{i,l_i}$ ($1 \leq i \leq n+1$). This allows to replace every variable 
$z_{i,j}$ by a concrete trace $v_{i,j}$. Statements
of the form $v_{i,j} I v_{k,l}$ and $v_{i,j} = v_{k,l}^{-1}$ can, of course, be eliminated.
Moreover, if there is an identity $y_{i,j} = v_{k,l}^{-1}$ then we can replace the variable $y_{i,j}$
by the concrete trace $v_{k,l}^{-1}$ (of length at most $\lambda$).

In the next step, we replace statements of the form $u_i^{x_i} = y_{i,1} \cdots y_{i,k_i}$  ($1 \leq i \leq n$).
Note that some of the variables $y_{i,j}$ might have been replaced by concrete traces  of length at most $\lambda$.
We apply to each of these equations Lemma~\ref{lemma-simplify-factorization-power}, or better
 Remark~\ref{remark-simplify-factorization-power}. This allows us to replace every equation $u_i^{x_i} = y_{i,1} \cdots y_{i,k_i}$
 ($1 \leq i \leq n$) by a disjunction of statements of the following form: There exist numbers
 $x_{i,j} > 0$ ($1 \leq i \leq n$, $j \in K_i$) such that
 \begin{itemize}
 \item $x_i = c_i +  \sum_{j \in K_i} x_{i,j}$ for all $1 \leq i \leq n$,
 \item $y_{i,j} = p_{i,j}  u_i^{x_{i,j}} s_{i,j}$ for all  $1 \leq i \leq n$, $j \in K_i$,
 \item $y_{i,j} = p_{i,j} s_{i,j}$ for all  $1 \leq i \leq n$, $j \in  [1,k_i] \setminus K_i$.
 \end{itemize}
Here, $K_i \subseteq [1,k_i]$, the $c_i$ are concrete numbers with $c_i \leq |A| \cdot (k_i-1)$, and the $p_{i,j}, s_{i,j}$ are concrete traces of length 
at most  $|A| \cdot (k_i-1) \cdot |u_i| \leq |A| \cdot (2^{2n+1}-3) \cdot \lambda$. Hence, the length of these traces can be exponential in $n$.

Note that since $x_i > 0$, we know the alphabet of  $y_{i,j} = p_{i,j}  u_i^{x_{i,j}} s_{i,j}$ (resp., $y_{i,j} = p_{i,j} s_{i,j}$).
This allows us to eliminate all independences of the form $y_{i,j} I y_{k,l}$ for $(i,j,k,l) \in J_1$ (see (c))
and  $y_{i,j} I z_{k,l}$ for $(i,j,k,l) \in J_2$ (see (d)). Note that all variables $z_{k,l}$ have already been replaced
by concrete traces. If $y_{i,j}$ was already replaced by a concrete trace, then we can determine from an equation
$y_{i,j} = p_{i,j}  u_i^{x_{i,j}} s_{i,j}$ the exponent $x_{i,j}$. Since $y_{i,j}$ was replaced by a trace of length at most $\lambda$ (a small number),
we get $x_{i,j} \leq \lambda$, and we can replace $x_{i,j}$ in $x_i = \sum_{j \in K_i} x_{i,j} + c_i$ by a concrete number of size at most $\lambda$.
Finally, if $y_{i,j}$ was replaced by a concrete trace, and we have an equation of the form
$y_{i,j} = p_{i,j} s_{i,j}$, then the resulting identity is either true or false and
can be eliminated.

After this step, we obtain a big disjunction of statements of the following form: There exist numbers
 $x_{i,j} > 0$ ($1 \leq i \leq n$, $j \in K'_i$) such that
\begin{enumerate}[(a')]
\item $x_i = c_i+ \sum_{j \in K'_i} x_{i,j}$ for all $1 \leq i \leq n$, and
\item $p_{i,j}  u_i^{x_{i,j}} s_{i,j}  = s_{k,l}^{-1} (u^{-1}_k)^{x_{k,l}} p_{k,l}^{-1}$
 for all $(i,j,k,l) \in M$.
\end{enumerate}
Here, $K'_i \subseteq K_i$ is a set of size at most $k_i \leq 2^{2n+1}-2$, $c_i \leq |A| \cdot (k_i-1) + \lambda \cdot k_i < (|A|+\lambda) \cdot (2^{2n+1}-2)$,
and the $p_{i,j}, s_{i,j}$ are concrete traces of length 
at most  $|A| \cdot (2^{2n+1}-3) \cdot \lambda$.  The set $M$ specifies a matching in the sense that for every exponent $x_{a,b}$ ($1 \leq a \leq n$, $b \in K'_i$)
there is a unique $(i,j,k,l) \in M$ such that $(i,j) = (a,b)$ or $(k,l) = (a,b)$. Note that
$$ 
|M| = \frac{1}{2} \sum_{i=1}^n |K'_i| \leq \frac{1}{2} \sum_{i=1}^n k_i \leq \frac{1}{2} (2^{2n+1}-2) = 2^{2n}-1.
$$
We now apply Lemma~\ref{lemma-connected-star} to the identities 
$p_{i,j}  u_i^{x_{i,j}} s_{i,j}  = s_{k,l}^{-1} (u^{-1}_k)^{x_{k,l}} p_{k,l}^{-1}$.
Each such identity can be replaced by a disjunction of constraints
$$
(x_{i,j}, x_{k,l}) \in \{ (a_{i,j,k,l} + b_{i,j,k,l} \cdot z_{i,j,k,l}, c_{i,j,k,l} + d_{i,j,k,l} \cdot z_{i,j,k,l}) \mid z_{i,j,k,l} \in \mathbb{N} \}.
$$
For the numbers $a_{i,j,k,l}, b_{i,j,k,l}, c_{i,j,k,l}, d_{i,j,k,l}$ we obtain the bound
$$
a_{i,j,k,l}, b_{i,j,k,l}, c_{i,j,k,l}, d_{i,j,k,l} \in O(\mu^8 \cdot \nu^{8 |A|})
$$
(the alphabet of the traces is $A^{\pm 1}$ which has size $2 |A|$, therefore, we have
to multiply in Lemma~\ref{lemma-connected-star} $|A|$ by $2$), where, by Lemma~\ref{lemma-prefixes},
\begin{equation} \label{eq-mu}
\mu = \max\{ \rho(p_{i,j}), \rho(p_{k,l}), \rho(s_{i,j}), \rho(s_{k,l})\} \in  O(|A|^\alpha \cdot 2^{2 \alpha n} \cdot \lambda^\alpha)
\end{equation}
and 
\begin{equation}  \label{eq-nu}
\nu = \max\{ \rho(u_i), \rho(u_k) \} \in O(\lambda^\alpha).
\end{equation}
Note that $\rho(t) = \rho(t^{-1})$ for every trace $t$. 
The above equation (a') for $x_i$ can be now written as
$$
x_i = c_i+ \sum_{(i,j,k,l) \in M} (a_{i,j,k,l} + b_{i,j,k,l} \cdot z_{i,j,k,l}) +  \sum_{(k,l,i,j) \in M} (c_{k,l,i,j} + d_{k,l,i,j} \cdot z_{k,l,i,j}) .
$$
Note that the two sums in this equation contain in total $|K'_i| \leq 2^{2n+1}$ many summands (since for every
$j \in K'_i$ there is a unique pair $(k,l)$ with $(i,j,k,l) \in M$ or $(k,l,i,j) \in M$). 

Hence, after a renaming of symbols, the initial equation \eqref{initial-equation} becomes equivalent to
a finite disjunction of statements of the form: There exist 
$z_1, \ldots, z_m \in \mathbb{N}$ (these $z_i$ are the above $z_{i,j,k,l}$ and $m = |M|$)  such that
\begin{equation} \label{eq-final-equation}
x_i = a_i + \sum_{j=1}^m  a_{i,j} z_j \text{ for all } 1 \leq i \leq n .
\end{equation}
Moreover, we have the following size bounds:
\begin{itemize}
\item $m = |M| \leq 2^{2n}-1$,
\item $a_i  \in O(c_i + |K'_i| \cdot \mu^8 \cdot \nu^{8 |A|})  \subseteq
O(2^{2n} (|A|+\lambda + \mu^8 \cdot  \nu^{8 |A|})) 
\subseteq O(2^{2n} \cdot \mu^8 \cdot  \nu^{8 |A|})$
\item $a_{i,j} \in O(\mu^8 \cdot  \nu^{8 |A|})$
\end{itemize}
Recall that some of the variables $x_i$ can be identical.
W.l.o.g. assume that $x_1, \ldots, x_k$ are pairwise different and for all $k+1 \leq i \leq n$, $x_i = x_{f(i)}$,
where $f : [k+1,n] \to [1,k]$. 
Then, the system of equations \eqref{eq-final-equation} is equivalent to 
\begin{gather*}
x_i = a_i + \sum_{j=1}^m  a_{i,j} z_j \text{ for all } 1 \leq i \leq k \\
a_i - a_{f(i)} = \sum_{j=1}^m  (a_{f(i),j}-a_{i,j}) z_j \text{ for all } k+1 \leq i \leq n .
\end{gather*}
The set of all $(x_1, \ldots, x_k) \in \mathbb{N}^k$ for which there exist $z_1, \ldots, z_m \in \mathbb{N}$ 
satisfying these equalities is 
semilinear by Lemma~\ref{lemma-linear-diophantine}, and if it is non-empty then it contains 
a vector
$(x_1, \ldots, x_k) \in \mathbb{N}^k$ such that
$$
x_i  \in O(n! \cdot m^2 \cdot 2^{2n(n+1)} \cdot \mu^{8(n+1)} \cdot  \nu^{8 |A| (n+1)})
\subseteq  O(n! \cdot  2^{2n(n+3)} \cdot \mu^{8(n+1)} \cdot  \nu^{8 |A| (n+1)}).
$$
Recall that in this bound we have to replace $n$ by $\alpha \cdot n$ due to the initial preprocessing.
This proves the theorem.
\end{proof}

\begin{theorem} \label{thm-main-graph-groups}
Let $(A,I)$ be a fixed independence alphabet.
Solvability of  compressed exponent equations over the graph group $\dG(A,I)$ is 
in {\sf NP}.
\end{theorem}

\begin{proof}
Consider a compressed exponent equation 
$$
E = (v_0 u_1^{x_1} v_1 u_2^{x_2}  v_2 \cdots u_n^{x_n} v_n = 1),
$$
where $u_i = \val(\cG_i)$ and $v_i = \val(\cH_i)$ for given SLPs $\cG_1, \ldots, \cG_n, \cH_0, \ldots, \cH_n$.
Let $m = \max\{ |\cG_1|, \ldots, |\cG_n|, |\cH_0|, \ldots, |\cH_n| \}$.
By Theorem~\ref{thm-main-technical} we know that if there exists a solution  for $E$ then 
there exists a solution $\sigma$ with 
$\sigma(x_i)  \in O((\alpha n)! \cdot  2^{2\alpha^2 n(n+3)} \cdot \mu^{8\alpha(n+1)} \cdot  \nu^{8 \alpha |A| (n+1)})$,
where
\begin{itemize}
\item $\mu \in  O(|A|^\alpha \cdot 2^{2 \alpha^2 n} \cdot \lambda^\alpha)$, 
\item $\nu  \in O(\lambda^\alpha)$,
\item $\lambda = \max\{ |u_1|, |u_2|, \ldots, |u_n|, |v_0|, |v_1|, \ldots, |v_n|\} \in 2^{O(m)}$, and
\item $\alpha \leq |A|$.
\end{itemize}
Note that the bound on the $\sigma(x_i)$ is exponential in the input length (the sum of the sizes
of all $\cG_i$ and $\cH_i$). 
Hence, we can guess in polynomial time the binary encodings of numbers 
$k_i \in O((\alpha n)! \cdot  2^{2\alpha^2 n(n+3)} \cdot \mu^{8\alpha(n+1)} \cdot  \nu^{8 \alpha |A| (n+1)})$
(where $k_i = k_j$ if $x_i = x_j$). Then, we have to verify whether 
$$\val(\cH_0) \val(\cG_1)^{k_1} \val(\cH_1) \val(\cG_2)^{k_2}  \val(\cH_2) \cdots \val(\cG_n)^{k_n} \val(\cH_n) = 1$$
in the graph group $\dG(A,I)$. This
is an instance of the so called {\em compressed word problem} for $\dG(A,I)$, where the input consists
of an SLP $\cG$ over the alphabet $A^{\pm 1}$ and it is asked whether $\val(\cG) =1$ in $\dG(A,I)$. 
Note that the big powers $\val(\cG_i)^{k_i}$ can be produced with the productions of $\cG_i$ and additional
$\lceil \log k_i \rceil$ many productions (using iterated squaring). Since the compressed word problem for a graph group can be solved
in deterministic polynomial  time \cite{Loh14,LoSchl07}, the statement of the theorem follows. For the last step, it 
is important that $(A,I)$ is fixed.
\end{proof}

\begin{remark}
Note that the bound on the exponents $\sigma(x_i)$ in the previous proof is still exponential  in the input length if
the  independence alphabet $(A,I)$ is part of the input as well. The problem is that we do not know whether
the {\em uniform compressed word problem} for graph groups (where the input is an independence alphabet $(A,I)$
together with an SLP over the terminal alphabet $A^{\pm 1}$) can be solved in polynomial time or at least in {\sf NP}.
The latter would suffice to get an {\sf NP}-algorithm for solvability of  compressed exponent equations over a graph group
that is part of the input.
\end{remark}

\section{Transfer results} \label{sec-transfer}
In this section, we show that the property of having an {\sf NP}-algorithm for
the knapsack problem (or compressed exponent equations) is preserved by certain
transformations on groups. Specifically, we show that the class of groups that
admit an {\sf NP}-algorithm for knapsack is closed under (i)~finite extensions,
(ii)~HNN-extensions with finite associated subgroups, and (iii)~amalgamated
free products with finite identified subgroups. In the case of finite extensions,
the transfer also holds for compressed exponent equations.

\paragraph{Finite extensions and virtually special groups.}
Our first transfer result concerns finite extensions. Together with our result
on graph groups, this will provide a large class of groups with an {\sf
NP}-algorithm for compressed exponent equations.  A group $G$ is called {\em
virtually special} if it is a finite extension of a subgroup of a graph group.
Recently, this class of groups turned out to be very rich.  It contains the
following classes of groups:
\begin{itemize}
\item Coxeter groups  \cite{HagWi10}
\item one-relator groups with torsion \cite{Wis09}
\item fully residually free groups \cite{Wis09}
\item fundamental groups of hyperbolic 3-manifolds \cite{Agol12}
\end{itemize}

\noindent
The following is our transfer theorem for finite extensions.
\begin{theorem} \label{thm-finite-index}
Let $G$ and $H$ be finitely generated groups such that $H$ is a finite
extension of $G$.  If knapsack (resp. solvability of compressed exponent
equations) belongs to {\sf NP} for $G$, then the same holds for $H$.
\end{theorem}

\noindent
From Theorem~\ref{thm-main-graph-groups} it follows that solvability of
compressed exponent equations belongs to {\sf NP} for every subgroup of a graph
group. Therefore, our transfer theorem implies:
\begin{theorem}
Solvability of compressed exponent equations belongs to {\sf NP} for every virtually special group. In particular,
solvability of compressed exponent equations belongs to {\sf NP} for Coxeter groups, one-relator groups with torsion, 
fully residually free groups, and
fundamental groups of  hyperbolic 3-manifolds.
\end{theorem}

\noindent
We need the following statement, which is shown implicitly in the proof of \cite[proof of Theorem~4.4]{Loh14}.
\begin{lemma} \label{lemma-cosets-compressed}
Let $G$ and $H$ be finitely generated groups such that $H$ is a finite extension of $G$
and let $C$  be a set of right coset representatives of $G$.
Let $A$ (resp. $B \supseteq A$) be a finite generating set for $G$ (resp., $H$).
From a given SLP $\cH$ over the terminal alphabet $B^{\pm 1}$ one can compute
in polynomial time (i) the unique coset representative $c \in C$ such that $\val(\cH) \in Gc$ 
and (ii) an SLP $\cG$ over the terminal alphabet $A^{\pm 1}$ such that
$\val(\cG) c = \val(\cH)$ holds in the group $H$.
\end{lemma}

\begin{proof}[Proof of Theorem~\ref{thm-finite-index}]
In \cite{KoeLoZe15}, it was shown that for each finitely generated group, the
knapsack problem and the solvability of exponential expressions where each
variable occurs only once (the latter is called \emph{generalized knapsack
problem} there) are polynomially inter-reducible. Therefore, we shall prove
that exponential expression over $H$ can be reduced to exponential expressions
over $G$.  Moreover, the reduction preserves the property that each variable
occurs only once. We only describe the case that all inputs are uncompressed;
by means of Lemma~\ref{lemma-cosets-compressed}, the compressed case can be
treated analogously.

Assume that $[H:G] = m$ and let $C$ be a set of coset representatives, $|C|=m$.
Let $A$ be a finite generating set for $G$ and let $B \supseteq A$ be a finite
generating set for $H$.  Suppose we are given an exponent equation
\begin{equation} v_0u_1^{x_1}v_1\cdots u_n^{x_n}v_n=1 \label{fi:input} \end{equation}
in $H$ where the $v_i$ and the $u_i$ are represented as words over $B^{\pm 1}$.
As a first step, we guess which of the variables $x_i$ assume a value smaller
than $m$.  For those that do, we can guess the value and merge the result in a
neighboring $v_i$. This increases the size of the instance by at most a factor
of $m$, which is a constant.  Hence, from now on, we only look for solutions to
(\ref{fi:input}) where $x_i\ge m$ for $1\le i\le n$.

The next step of our ${\sf NP}$ algorithm is to guess the cosets 
occurring in a solution. This means, we guess $d_0,c_1,d_1,\ldots,c_n,d_n\in C$
and look for a solution to (\ref{fi:input}) such that
$v_0u_1^{x_1}v_1\cdots u_i^{x_i}v_i\in Gd_i$ and $v_0u_1^{x_1}v_1\cdots u_i^{x_i}\in Gc_i$ for $0\le i\le n$.
This is equivalent to a solution where the elements
\begin{equation*}
v_0d_0^{-1},\quad d_{i-1}u_i^{x_i}c_i^{-1},\quad c_{i}v_id_i^{-1},\quad d_n
\end{equation*}
all belong to $G$ for $1\le i\le n$. We can verify in polynomial time that
$v_0d_0^{-1}$, $c_iv_id_i^{-1}$ ($1\le i\le n$), and $d_n$ belong to $G$.
Therefore, we want to check whether there is a solution to (\ref{fi:input})
where $d_{i-1}u_i^{x_i}c_i^{-1}\in G$ for $1\le i\le n$. 

Consider the function $f_i\colon C\to C$, which is defined so that for each
$c\in C$, $f_i(c)$ is the unique element $d\in C$ with $cu_id^{-1}\in G$. Note
that we can compute $f_i$ in polynomial time. Then there are numbers $1\le k_i\le
m$ such that $f_i^{m+k_i}(d_{i-1})=f_i^m(d_{i-1})$.  With this notation, we
have $d_{i-1}u_i^{x_i}c_i^{-1}\in G$ if and only if $f_i^{x_i}(d_{i-1})=c_i$.

We may assume that there is an $x_i\ge m$ with $f_i^{x_i}(d_{i-1})=c_i$:
Otherwise, there is no solution and we can terminate our branch. Therefore,
there is a $0\le r_i<k_i$ such that $f_i^{m+r_i}(d_{i-1})=c_i$. This means, we
have $f_i^{x_i}(d_{i-1})=c_i$ for $x_i\ge m$ if and only if $x_i=m+k_i\cdot
y_i+r_i$ for some $y_i\ge 0$.  This allows us to construct an exponent equation
over $G$.

Let $e_i=f_i^m(d_{i-1})$. Then, the elements $d_{i-1}u_i^me_i^{-1}$,
$e_iu_i^{k_i}e_i^{-1}$, and $e_iu_i^{r_i}c_i^{-1}$ all belong to $G$. Moreover,
for $x_i=m+k_i\cdot y_i+r_i$, we have
\begin{align*}
v_0u_1^{x_1}v_1\cdots u_n^{x_n}v_n &= v_0d_0^{-1} \prod_{i=1}^n d_{i-1} u_i^{m+k_i\cdot y_i+r_i}c_i^{-1}c_i v_i d_i^{-1} \\
&=(v_0d_0^{-1}) \prod_{i=1}^n (d_{i-1} u_i^m e_i^{-1})(e_iu_i^{k_i} e_i^{-1})^{y_i} (e_i u_i^{r_i} c_i^{-1}c_i v_i d_i^{-1})
\end{align*}
and each term in parentheses belongs to $G$. This clearly yields an exponent
equation over $G$ (with variables $y_1,\ldots,y_n$) that is solvable if and
only if there is a solution of (\ref{fi:input}) of the kind we are looking for.
It remains to verify that the new instance is polynomial in size. 

There is a constant $\ell$ such that given a word $w$ representing $h\in H$ and
elements $c,d\in C$ such that $chd^{-1}\in G$, a word of length at most
$\ell\cdot|w|$ representing $chd^{-1}$ over $A^{\pm 1}$ is computable in linear
time.  Let $s_i,t_j\in (B^{\pm 1})^*$ represent $v_i$ and $u_j$, respectively,
for $0\le i\le n$ and $1\le j\le n$.  Then, the new instance has size at most
\[ \ell|s_0|+\sum_{i=1}^n \ell(m+k_i+r_i)|t_i|+\ell|s_i| \le 3m\ell (|s_0|+|t_1|+|s_1|+\cdots |t_n|+|s_n|) \]
which is linear in the size of the old instance.
\end{proof}

\paragraph{HNN-extensions and amalgamated products.}
The remaining transfer results concern two constructions that are of
fundamental importance in combinatorial group theory~\cite{LySch77}, namely
HNN-extensions and amalgamated products. In their general form, HNN-extensions
have been used to construct groups with an undecidable word problem, which
means they may destroy desirable algorithmic properties. We consider the
special case of finite associated (resp. identified) subgroups, for which these
constructions already play a prominent role, for example, in Stallings'
decomposition of groups with infinitely many ends~\cite{Stal71} or the
construction of virtually free groups~\cite{DiDu89}.  Moreover, these
constructions are known to preserve a wide range of important structural and
algorithmic
properties~\cite{AllGre73,Bez98,HaLo11,KaSiSt06,KaWeMy05,KaSo70,KaSo71,LoSe06,LoSe08,MeRa04}.

Suppose $G=\langle \Sigma\mid R\rangle$ is a finitely generated group that has
two isomorphic subgroups $A$ and $B$ with an isomorphism $\varphi\colon A\to
B$. Then the corresponding \emph{HNN-extension} is the group \[ H=\langle G, t
\mid t^{-1}at=\varphi(a)~(a\in A)\rangle, \] where $t$ is a new letter not
contained in $G$. In other words, $H$ is the group $H=\langle \Sigma\cup \{t\}
\mid R\cup \{t^{-1}at=\varphi(a) \mid a\in A\}\rangle$ with $t\notin\Sigma$.
Intuitively, $H$ is obtained from $G$ by adding a new element $t$ such that
conjugating elements of $A$ with $t$ applies the isomorphism $\varphi$.  Here,
$t$ is called the \emph{stable letter} and the groups $A$ and $B$ are the
\emph{associated subgroups}.   A basic fact about HNN-extensions is that
the group $G$ embeds naturally into $H$~\cite{HiNeNe49}.

Here, we only consider the case that $A$ and $B$ are finite groups, so that we
may assume that $A\cup B\subseteq\Sigma$. To exploit the symmetry of the
situation, we use the notation $A(+1)=A$ and $A(-1)=B$. Then, we have
$\varphi^{\alpha}\colon A(\alpha)\to A(-\alpha)$ for $\alpha\in\{+1,-1\}$.  By
$h\colon (\Sigma^{\pm 1}\cup\{t,t^{-1}\})^*\to H$, we denote the canonical
morphism that maps each word to the element of $H$ it represents. 

A word $u\in (\Sigma^{\pm 1}\cup \{t,t^{-1}\})^*$ is called \emph{reduced} if
it does not contain a factor $t^{-\alpha} w t^\alpha$ with $\alpha\in\{-1,1\}$,
$w\in(\Sigma^{\pm 1})^*$, and $h(w)\in A(\alpha)$. Note that the equation
$t^{-1}at=\varphi(a)$, $a\in A$, allows us to replace such a factor
$t^{-\alpha}wt^{\alpha}$ by $\varphi^{\alpha}(h(w))\in A(-\alpha)\subseteq\Sigma$.  Since this
reduces the number of $t$'s in the word, this allows us to turn every word into
an equivalent reduced word. The following well-known fact describes the reduced words
representing the identity~\cite[Lemma 5]{LoSe08}.

\begin{lemma}\label{hnn:reducedidentity}
If $u\in(\Sigma^{\pm 1}\cup\{t,t^{-1}\})^*$ is a reduced word representing
$1\in H$, then $u\in(\Sigma^{\pm 1})^*$.
\end{lemma}

Our algorithm for knapsack in HNN-extensions is an adaptation of the saturation
algorithm of Benois~\cite{Benois69} for the membership problem for rational
subsets of free groups.  Here, for each path spelling $aa^{-1}$, one adds a
parallel edge labeled with the empty word.  Since knapsack is a special case of
this problem, we have to use a suitable subclass of automata that is preserved
by our saturation and corresponds to the knapsack problem.

Let $G$ be a group with finite generating set $\Sigma$. A \emph{finite
automaton over $G$} is an NFA $\cA=(Q,\Sigma^{\pm 1}, \Delta, q_0, F)$. A
\emph{(directed) cycle} in $\cA$ is a sequence $p_1,\ldots,p_n$ of states such
that there are edges $(p_i, a_i, p_{i+1})$ for $1\le i\le n-1$ and
$(p_n, a_n, p_1)$ with $a_1,\ldots,a_n\in\Sigma^{\pm 1}$. In
particular, a single state with a loop is regarded as a cycle.  A sequence
$p_1,\ldots,p_n$ is an \emph{induced cycle} if it is a cycle and there are no
other edges among the states $p_1,\ldots,p_n$.  We call $\cA$ a \emph{knapsack
automaton} if every strongly connected component of $\cA$ is a singleton or an
induced cycle. The \emph{membership problem for knapsack automata over $G$} is
the following decision problem:

\smallskip
\noindent
{\bf Input:} A knapsack automaton $\cA$ over $G$ and a word $w\in(\Sigma^{\pm 1})^*$.

\smallskip
\noindent
{\bf Question:} Does $\cA$ accept a word $w'$ that represents the same element of $G$ as $w$?

\smallskip
\noindent
Indeed, the membership problem for knapsack automata corresponds precisely to
the knapsack problem in the following sense.

\begin{lemma}\label{knapsack-automata}
For each finitely generated group, knapsack belongs to {\sf NP} if and only if
membership for knapsack automata belongs to {\sf NP}.
\end{lemma}
\begin{proof}
It is easy to turn a knapsack instance into a knapsack automaton: Given words
$w_1,\ldots,w_k,w\in (\Sigma^{\pm 1})^*$, one can clearly construct a knapsack
automaton accepting $w_1^*\cdots w_k^*$. Then, the knapsack problem amounts to
deciding the membership problem for $w$.

Now, suppose we are given a knapsack automaton $\cA$ over $G$ and a word
$w\in(\Sigma^{\pm 1})^*$.  We can clearly turn $\cA$ into a knapsack automaton
that first reads $w^{-1}$ and then behaves like $\cA$.  Therefore, it suffices
to solve the membership problem in the case that $w=\varepsilon$.

Consider a run $r$ in $\cA$ from the initial to a final state.  Let
$c_1,\ldots,c_n$ be the sequence of strongly connected components it visits.
For each $c_i$ that is not a singleton, let $p_i$ and $q_i$ be the state where
$r$ enters and leaves $c_i$, respectively. We call the sequence
$c_1,\ldots,c_n$, together with the $p_i$ and $q_i$ the \emph{skeleton} of $r$.

Our algorithm guesses a skeleton.  Since $\cA$ is a knapsack automaton, from
this skeleton, we can determine words
$v_0,u_1,v_1,\ldots,u_n,v_n\in(\Sigma^{\pm 1})^*$ such that 
$v_0u_1^*v_1\cdots u_n^*v_n$ is precisely the set of words labeling a path
with this skeleton. Hence, deciding
the membership problem for $\cA$ amounts to checking whether there are
$x_1,\ldots,x_n\in\mathbb{N}$ with $h_0g_1^{x_1}h_1\cdots g_n^{x_n}h_n=1$,
where $g_i$ ($h_j$, respectively) is the element represented by $u_i$ ($v_j$,
respectively).  This is an exponential equation with pairwise distinct variables
and the solvability of such equations is called the \emph{generalized knapsack
problem} in~\cite{KoeLoZe15}, where it was shown to be polynomially
inter-reducible with the knapsack problem.
\end{proof}

\begin{theorem}\label{thm-hnn}
Let $H$ be an HNN-extension of the finitely generated group $G$ with finite
associated subgroups.  If knapsack for $G$ belongs to {\sf NP}, then the same
holds for $H$.
\end{theorem}

\begin{proof}
According to Lemma~\ref{knapsack-automata}, it suffices to prove that if
membership for knapsack automata over $G$ belongs to {\sf NP}, then the same holds
for $H$. Hence, let $\cA$ be a knapsack automaton over $H$. As explained above,
it suffices to check membership for the group identity, i.e., to check whether $\cA$
accepts a word from $h^{-1}(1)$.

The basic idea of the proof is to saturate $\cA$, yielding a knapsack automaton
that is \emph{saturated}, meaning: For each path from $p$ to $q$ labeled with a
word $t^{-\alpha}wt^\alpha$ with $h(w)\in A(\alpha)$, there is an edge from $p$
to $q$ labeled with $\varphi^\alpha(h(w))\in A(-\alpha)$. We will then show that
$\cA$ accepts a word from  $h^{-1}(1)$ if and only if it accepts a word from  
$h^{-1}(1) \cap (\Sigma^{\pm 1})^*$.  
This will allow us to remove
all $t^{\pm 1}$-edges and apply the algorithm for $G$. A path in a
knapsack automaton that is labeled by a word $t^{-\alpha}wt^\alpha$ with $w\in
(\Sigma^{\pm 1})^* \cap h^{-1}(A(\alpha))$ is called  a \emph{reduction path}.
Among other things, the algorithm will introduce a \emph{shortcut edge} for the
reduction path, namely
\begin{equation}
p\xrightarrow{\varphi^{\alpha}(a)}q, \label{hnn:shortcut}
\end{equation}
where $a = h(w) \in A(\alpha)$.
Observe that $\varphi^{\alpha}(a)\in A(-\alpha)$ and
$\varphi^\alpha(a)=h(t^{-\alpha}wt^\alpha)$.  By introducing intermediate
states, we may assume that (i)~there is no edge between states that belong to
distinct cycles and (ii)~the initial and the final state do not lie on a cycle.

\paragraph{Phase 1.} The saturation proceeds in two phases. In the first phase, we
saturate the directed cycles, which are the strongly connected components. This
means, we modify the automaton so that there is no reduction path between two
states on a cycle. This is done as follows. We successively guess tuples
$(p,\alpha,a,q)$ where $p$ and $q$ are states from the same cycle,
$\alpha\in\{-1,1\}$, and $a\in A(\alpha)$. Then, employing the {\sf NP}
algorithm for $G$, we can clearly verify that there is a reduction path
spelling $t^{-\alpha}wt^\alpha$ from $p$ to $q$ with $w\in (\Sigma^{\pm 1})^*$ and 
$h(w)=a$. Note that on this path, the first letter ($t^{-\alpha}$) occurs only
once, meaning the path visits each state at most once (i.e. it makes at most
one round in the cycle).  Let
\begin{equation} p=r_0 \xrightarrow{u_1} r_1 \cdots \xrightarrow{u_n} r_n=q \label{hnn:redpath}\end{equation}
be the reduction path and let
\[ q=r_n\xrightarrow{u_{n+1}}r_{n+1}\cdots \xrightarrow{u_m} r_m=r_0=p \]
be the rest of the cycle with $u_1,\ldots,u_m\in\Sigma^{\pm
1}\cup\{t,t^{-1}\}$ and $u_1\cdots u_n=t^{-\alpha}wt^\alpha$. 
In particular, $m$ is the length of the cycle. Let us now
describe the saturation step.  We remove all edges from (\ref{hnn:redpath}) and
all states incident to them, except for $p$ and $q$.  Instead, we add a shortcut edge (\ref{hnn:shortcut}).
For each state $s$ not on the cycle and
for which there is an edge $(s,v,r_i)$, $1\le i\le n-1$, we glue in a path
\begin{equation} s\xrightarrow{v}s_0\xrightarrow{u_{i+1}}s_1\cdots \xrightarrow{u_n}s_{n-i}=q, \label{hnn:states:arriving} \end{equation}
where $s_0,\ldots,s_{n-i-1}$ are new states. Analogously, for each state $s$
not on the cycle and for which there is an edge $(r_i,v,s)$, $1\le i\le n-1$,
we glue in a path
\begin{equation} p=s_0\xrightarrow{u_1}s_1\cdots\xrightarrow{u_i} s_i\xrightarrow{v} s,  \label{hnn:states:leaving}\end{equation}
where $s_1,\ldots,s_i$ are new states. Moreover, for each pair $(s,s')$ of
states not on the cycle and for which there are edges $(s,v,r_i)$ and
$(r_j,v',s')$ with $1\le i<j\le n-1$, we glue in a path
\begin{equation} s \xrightarrow{v} s_0 \xrightarrow{u_{i+1}} s_{1} \cdots \xrightarrow{u_j} s_{j-i} \xrightarrow{v'} s', \label{hnn:states:pairs} \end{equation}
where $s_0,\ldots,s_{j-i}$ are new states.  This completes our saturation step.

Let $\cA'$ be the automaton resulting from one saturation step from $\cA$.
Then, $\cA'$ is clearly a knapsack automaton: We only connect states that were
connected before. Moreover, for states $s,s'$ that exists in $\cA$ and in
$\cA'$, the set of group elements represented on paths from $s$ to $s'$ does
not change. Indeed, a path that avoids our cycle still exists.  A path that
involves the whole path (\ref{hnn:redpath}) can use the shortcut edge
(\ref{hnn:shortcut}). A path that either (i)~enters (\ref{hnn:redpath})
after $p$ and follows it until $q$ or (ii)~follows (\ref{hnn:redpath}) partly
and then leaves before $q$ can use the new paths (\ref{hnn:states:arriving}) or
(\ref{hnn:states:leaving}), respectively. Finally, a path that follows only a
part of (\ref{hnn:redpath}) that starts after $p$ and ends before $q$ can use
the new path (\ref{hnn:states:pairs}) instead.

Let us estimate the number of added states during Phase 1. The \emph{degree} of
a cycle is the number of edges entering or leaving the cycle. Let $d$ be the
degree of our cycle. Let us first consider a single saturation step. The new
states of type (\ref{hnn:states:arriving}) or (\ref{hnn:states:leaving}) are
each at most $d\cdot n$ many.  The new states of type (\ref{hnn:states:pairs})
are at most $d^2\cdot n$ many. Hence, we add at most $(d^2+2d)n\le (d^2+2d)m$
states in this saturation step. Observe that in this step, the length of the
affected cycle decreases ($t^{-\alpha}wt^\alpha$ has length $\ge 2$ and
$h(t^{-\alpha}wt^\alpha)\in A(-\alpha)$ has length $1$) and its degree is
unchanged (the new edges from
(\ref{hnn:states:arriving}) and (\ref{hnn:states:leaving}) clearly preserve the
degree and those of (\ref{hnn:states:pairs}) do not increase the degree because
by our assumption that no edge connects two cycles, $s$ and $s'$ do not belong to a cycle). Now, we consider the
whole phase. Suppose in the beginning, $\cA$ has $c$ cycles of maximal degree
$d$ and maximal length $\ell$. Then, each saturation step adds at most
$(d^2+2d)\ell$ states.  Moreover, there can be at most $\ell\cdot c$ saturation
steps, so that the first phase adds at most $(d^2+2d)\ell^2c$ states, which is
polynomial in the size of the input automaton.

\paragraph{Phase 2.} In the second phase, we consider reduction paths between states
that belong to distinct strongly connected components. Since here, adding an
edge that runs parallel to the reduction path cannot violate the property of
being a knapsack automaton, we may saturate by simply introducing new edges.

Again, we successively guess tuples $(p,\alpha,a,q)$ where $\alpha\in\{-1,1\}$,
and $a\in A(\alpha)$. However, we require that $p$ and $q$ are not from the
same strongly connected component and that there is no shortcut edge
(\ref{hnn:shortcut}) yet. As above, we employ the {\sf NP} algorithm for $G$ to
verify that there is a reduction path spelling $t^{-\alpha}wt^\alpha$ from $p$
to $q$ with $w\in (\Sigma^{\pm 1})^*$ and $h(w)=a$. Then, we add the shortcut
edge~(\ref{hnn:shortcut}).  As before, we have
$h(t^{-\alpha}wt^\alpha)=\varphi^\alpha(a)\in A(-\alpha)$.  This is all we do
in the saturation step. Since now, we only add edges (and no states) and each
correct guess leads to an increase in the number of edges, our sequence of
saturation steps must terminate after a polynomial number of steps. This
concludes the second phase and thus the saturation.

\medskip
\noindent
Finally, the algorithm applies the {\sf NP}-algorithm for $G$. More precisely,
we remove all edges labeled $\{t,t^{-1}\}$. This yields a knapsack automaton
over $G$, so that we can use the algorithm for $G$ to check whether it accepts
$1\in G$. Then, we answer ``yes'' if and only if the algorithm for $G$ does.

It remains to be shown that this algorithm is sound and complete. If we answer
``yes'', then the input automaton accepts $1\in H$. This is because each
saturation step preserves the set of accepted elements. On the other hand,
suppose the input automaton $\cA$ accepts $1\in H$ and consider the branch of
our nondeterministic algorithm that guesses in such a way that in the end, there are no more
reduction paths without a shortcut edge. Let $\cB$ be the resulting saturated
knapsack automaton. Since $\cA$ accepts $1\in H$, there is an accepting run in
$\cB$ that accepts $1\in H$. Consider such a run reading a word $u\in
(\Sigma^{\pm 1}\cup \{t,t^{-1}\})^*$ with a minimal number of occurrences of
$t$. Since $\cB$ is saturated, this implies that $u$ is reduced: Otherwise, $u$
would have a factor $t^{-\alpha}wt^{\alpha}$ with $w\in (\Sigma^{\pm 1})^*$ and
$h(w)\in A(\alpha)$.  This factor, however, lies on a reduction path and we
could have used the shortcut edge instead, which would result in a run with
fewer $t$'s. Since $u$ is reduced and represents $1\in H$, it contains neither
$t$ nor $t^{-1}$ (Lemma~\ref{hnn:reducedidentity}).  Hence, our application of
the algorithm for $G$ answers ``yes'' because of $u$.
\end{proof}

\noindent
In our last transfer theorem, we consider amalgamated free products. For
$i\in\{0,1\}$, let $G_i=\langle \Sigma_i \mid R_i\rangle$ be a finitely
generated group  and let $F$ be a finite group that is embedded in each $G_i$,
meaning that there are injective morphisms $\varphi_i\colon F\to G_i$ for
$i\in\{0,1\}$.    Then, the \emph{free product with amalgamation with
identified subgroup $F$} is defined as
\[ G_0 *_F G_1 = \langle G_0*G_1 \mid \varphi_0(f)=\varphi_1(f)~(f\in F)\rangle. \]
Here, $G_0*G_1$ denotes the free product $G_0*G_1=\langle \Sigma_0\uplus \Sigma_1 \mid R_0\uplus
R_1\rangle$.  Note that the product depends on the morphisms $\varphi_i$,
although they are omitted in the notation $G_0 *_F G_1$.  Equivalently, $G_0 *_F G_1$
is given by the presentation 
\[ \langle \Sigma_0\uplus \Sigma_1 \mid R\uplus S\cup \{\varphi_0(f)=\varphi_1(f) \mid f\in F\}\rangle. \]

Let us consider the free product $G_0*G_1$.  Let $h\colon (\Sigma_0^{\pm
1}\cup\Sigma_1^{\pm 1})^*\to G_0 * G_1$ be the canonical morphism that maps a
word to the group element it represents.  If $w\in (\Sigma_0^{\pm
1}\cup\Sigma_1^{\pm 1})^*$, then a \emph{syllable of $w$} is a factor of $w$
that is contained in $(\Sigma_0^{\pm 1})^+\cup (\Sigma_1^{\pm 1})^+$ and that
is maximal with this property.  The definition of the free product immediately
implies the following.
\begin{lemma}\label{amalg:identity}
If in the free product $G_0*G_1$, a word represents $1\in G_0*G_1$, then it
contains a syllable $s$ with $h(s)=1$. 
\end{lemma}
The transfer theorem states that taking amalgamated products with finite
identified subgroups preserves {\sf NP} membership of knapsack.
\begin{theorem}
Let $G_0$ and $G_1$ be finitely generated groups with a common finite subgroup $F$.
If knapsack for $G_0$ and for $G_1$ belongs to {\sf NP}, then the same holds for
the amalgamated product $G_0 *_F G_1$.
\end{theorem}
\begin{proof}
It is well-known~\cite[Theorem~2.6, p.~187]{LySch77} that $G_0 *_F G_1$ can be
embedded into the HNN-extension
\[ I = \langle G_0*G_1, t \mid t^{-1}\varphi_0(f)t=\varphi_1(f)~(f\in F) \rangle \]
by way of the morphism $\Phi\colon G_0 *_F G_1\to I$ with
\[ \Phi(g)=\begin{cases} t^{-1}gt & \text{if $g\in G_0$} \\ g & \text{if $g\in G_1$}. \end{cases} \]
Since Theorem~\ref{thm-hnn} already tells us that {\sf NP} membership of
knapsack is preserved by HNN-extensions with finite associated subgroups, it
suffices to show that free products preserve {\sf NP} membership.

We use a slight modification of the nondeterministic algorithm from the proof
of Theorem~\ref{thm-hnn} and show that if membership for knapsack automata
belongs to {\sf NP} for $G_0$ and $G_1$, the same holds for $G_0 * G_1$.
During the saturation, we maintain the following invariants:
\begin{enumerate}[(i)]
\item There is no edge between states
that belong to distinct cycles.
\item The initial and the final states do not
lie on a cycle.
\item Every edge entering a cycle is labeled
with the empty word $\varepsilon$.
\end{enumerate}
By introducing intermediate states, we can clearly achieve them in the
beginning.  As in the proof of Theorem~\ref{thm-hnn}, we add shortcut edges for
reduction paths. For states $p$ and $q$, a \emph{reduction path (from $p$ to
$q$)} is a path labeled by a word $w\in (\Sigma_i^{\pm 1})^+$ for some $i\in
\{0,1\}$ such that (a)~$h(w)=1$ and (b)~if $p$ and $q$ lie on a cycle, then
this cycle also contains a letter in $\Sigma_{1-i}^{\pm 1}$.  Here, we need the
additional condition (b) to make sure that short-cutting a reduction path
actually reduces the cycle (Without requiring (b), it could happen that a
reduction path occupies more than one round of a cycle.) A \emph{shortcut edge}
is then simply $(p,\varepsilon,q)$.

Again, our saturation consists of two phases and in the first one, we shortcut
reduction paths inside of cycles. We guess tuples $(p,i,q)$ such that $p$ and
$q$ lie on a cycle and $i\in\{0,1\}$.  Using the {\sf NP}-algorithm for $G_i$,
we verify that there is a reduction path from $p$ to $q$ labeled with
$w\in(\Sigma_i^{\pm 1})^+$. Then, we proceed as in the proof of
Theorem~\ref{thm-hnn} and replace the reduction path with a shortcut edge and
add new paths almost as in (\ref{hnn:states:arriving}),
(\ref{hnn:states:leaving}), and (\ref{hnn:states:pairs}): The only difference
is that the new paths of type (\ref{hnn:states:arriving}) are are prolonged
with an $\varepsilon$-edge at the end so as to preserve invariant (iii).

While in the proof of Theorem~\ref{thm-hnn}, the length the cycle decreases in
a saturation step, this is not guaranteed here. This is because in the proof of
Theorem~\ref{thm-hnn}, we always remove edges labeled $t$ and $t^{-1}$. Here,
it could happen that the reduction path consists of one edge labeled
$a\in\Sigma_i^{\pm 1}$ with $h(a)=1$. Then, the length of the cycle is
unchanged. We do, however, reduce the number of letters on the cycle.
Therefore, an analogous estimation of the number of introduced states applies
and shows that it is polynomially bounded.

The second phase works just as for Theorem~\ref{thm-hnn}. We guess triples
$(p,i,q)$ such that $p$ and $q$ are not in the same strongly connected
component but there is no shortcut edge $(p,\varepsilon,q)$ yet.
Then, we verify that there is a reduction path from $p$ to $q$ with label
$w\in(\Sigma_i^{\pm 1})^+$. If this is the case, we add a shortcut edge
$(p,\varepsilon,q)$.

In the end, we guess $i\in\{0,1\}$ and verify, using the {\sf NP}-algorithm for
$G_i$, that the automaton, restricted to $\Sigma_i^{\pm 1}$, accepts a word
representing $1\in G_i$.

Let us show that this algorithm is sound and complete. As above, we can argue
that if it answers ``yes'', then the input automaton clearly accepts $1\in
G_0*G_1$.  For the completeness, we have to argue slightly differently. Suppose
the input automaton accepts a word representing $1\in G_0 * G_1$. We consider a
branch of the nondeterministic algorithm that saturates every reduction path.
Let $\cB$ be the resulting automaton.  Since $\cB$ also accepts a word
representing $1\in G_0 * G_1$, we consider such a word $w\in (\Sigma_0^{\pm
1}\cup \Sigma_1^{\pm 1})^*$ with a minimal number of syllables.

Suppose $w$ has more than one syllable. By Lemma~\ref{amalg:identity}, it
contains a syllable $s\in (\Sigma_i^{\pm 1})^+$ with $h(s)=1$.  Consider the
accepting run $r$ for $w$ and let $p$ and $q$ be the states occupied before and
after reading $s$.  The path taken by $r$ from $p$ to $q$ is not a reduction
path, because otherwise we could have taken a shortcut edge instead, in
contradiction to the minimality of $w$. This means, $p$ and $q$ lie on a cycle
that contains only letters in $\Sigma_i^{\pm 1}$.  Since $s$ is a syllable,
this implies that $r$ enters this cycle at $p$. Let $p'$ be the state occupied
in $r$ directly before $p$: Note that $r$ cannot start in $p$ because of
invariant (ii).  Because of invariant (iii), the edge from $p'$ to $p$ is
labeled with $\varepsilon$.  Thus, the path taken by $r$ from $p'$ to $q$ is a
reduction path, again contradicting the minimality of $w$.

Hence, $w$ has at most one syllable, which means $w\in (\Sigma_j^{\pm 1})^*$ for
some $j\in\{0,1\}$ and our application of the {\sf NP}-algorithm for $G_j$
answers ``yes''.
\end{proof}

\section{Hardness results}

Since knapsack for binary encoded integers is {\sf NP}-complete, it follows that the compressed
knapsack problem is {\sf NP}-hard for every group that contains an element of infinite order. 
In this section, we prove that (uncompressed) knapsack and subset sum are {\sf NP}-complete for a direct 
product of two free groups of rank at least two. This solves an open problem from \cite{FrenkelNU15}.

With $F(\Sigma)$ we denote  the free group generated by the set $\Sigma$. Moreover, let $F_2 = F(\{a,b\})$.

\begin{theorem}
The subset sum problem and the knapsack problem are {\sf NP}-complete for $F_2 \times F_2$.
For knapsack {\sf NP}-hardness already holds for the variant, where the exponent variables are
allowed to take values from $\mathbb{Z}$ (see Remark~\ref{remark}).
\end{theorem}

\begin{proof}
In \cite{MyNiUs14} it was shown that there exists a fixed set $D \subseteq F_2 \times F_2$ such that that the following problem
(called the bounded submonoid problem)
is {\sf NP}-complete:

\smallskip
\noindent
{\bf Input:} A unary encoded number $n$ (i.e., $n$ is given by the string $a^n$) and an element $g \in F_2 \times F_2$

\smallskip
\noindent
{\bf Question:}  Do there exist $g_1, \ldots g_n \in D$ (not necessarily distinct) such that $g = g_1 g_2 \cdots g_n$ in $F_2 \times F_2$?

\smallskip
\noindent
Let us briefly explain the {\sf NP}-hardness proof, since we will reuse it.
We start with a finitely presented group $\langle\Sigma, R \rangle$ having
an {\sf NP}-complete word problem and a polynomial Dehn function. Such a group was
constructed in \cite{BiOlSa02}. To this group, the following classical construction by Miha{\u\i}lova \cite{Mih66} 
is applied: Let 
$$
D  = \{ (r^\epsilon,1) \mid r \in R, \epsilon \in \{-1,1\} \} \cup \{ (a,a) \mid a \in \Sigma^{\pm 1} \},
$$
which is viewed as a subset of $F(\Sigma) \times F(\Sigma)$. Note that $D$ is closed under taking inverses.
Let $\langle D \rangle \leq F(\Sigma) \times F(\Sigma)$ be the subgroup generated by $D$.
Miha{\u\i}lova proved that for every word $w \in (\Sigma^{\pm 1})^*$ the following equivalence holds:
$$
w = 1 \text{ in } \langle \Sigma, R \rangle \ \Longleftrightarrow \ (w,1) \in \langle D \rangle \text{ in } F(\Sigma) \times F(\Sigma) .
$$
Moreover, based on the fact that $\langle \Sigma, R \rangle$ has a polynomial Dehn function $p(n)$,
the following equivalence was shown in \cite{MyNiUs14}, 
where $q(n) = p(n) + 8(c \cdot p(n) + n)$, $c$ is the maximal length of a relator in $R$, and $D^n$ is the 
set of all products of $n$ elements from $D$:
$$
w = 1 \text{ in } \langle \Sigma, R \rangle \ \Longleftrightarrow \  \exists n \leq q(|w|) : 
(w,1)  \in D^n   \text{ in } F(\Sigma) \times F(\Sigma) .
$$
From these two equivalences it follows directly that the following three statements are equivalent for all words
$w \in (\Sigma^{\pm 1})^*$, where $D = \{ g_1, g_2, \ldots, g_k\}$:
\begin{itemize}
\item $w = 1$ in $\langle\Sigma, R \rangle$
\item $(w,1) = \prod_{i=1}^{q(|w|)} (g_1^{a_{1,i}} g_2^{a_{2,i}} \cdots g_k^{a_{k,i}})$ in $F(\Sigma) \times F(\Sigma)$ for $a_{j,i} \in \{0,1\}$
\item $(w,1) = \prod_{i=1}^{q(|w|)} (g_1^{a_{1,i}} g_2^{a_{2,i}} \cdots g_k^{a_{k,i}})$ in $F(\Sigma) \times F(\Sigma)$ for $a_{j,i} \in \mathbb{Z}$
\end{itemize}
This shows that the subset sum problem and the knapsack problem are {\sf NP}-hard for the group 
$F(\Sigma) \times F(\Sigma)$, where for knapsack we allow integer exponents.
 To get the same results for $F_2 \times F_2$, we use the fact that
$F_2$ contains a copy of $F(\Sigma)$.
\end{proof}

\bibliographystyle{abbrv}
\bibliography{bib}

\end{document}